\documentclass[11pt,a4paper]{amsart}
\usepackage[utf8]{inputenc}
\usepackage[T1]{fontenc}
\usepackage[english]{babel}

\usepackage{sidecap}

\usepackage{verbatim}
\usepackage{lmodern}
\usepackage{amsmath}
\usepackage{amssymb} 
\usepackage{amsthm} 
\usepackage{thmtools}
\usepackage{enumitem}
\usepackage{graphicx}
\usepackage[colorlinks=true, linkcolor={blue!50!black}, pdfhighlight=/O, 
ocgcolorlinks=true]{hyperref}
\hypersetup{bookmarksdepth=2}
\usepackage{mathtools}
\usepackage{indentfirst}
\usepackage{tabularx}
\usepackage{bbm}
\usepackage[capitalise,nosort]{cleveref}
\usepackage{stmaryrd}
\usepackage[margin=.94in]{geometry}
\usepackage{wasysym}

\usepackage{tikz} 
\usepackage{tikz-cd} 
\usetikzlibrary{arrows,calc,decorations.markings}
\usepackage{xparse}

\usepackage{blindtext}

\newcommand*\cocolon{%
        \nobreak
        \mskip6mu plus1mu
        \mathpunct{}%
        \nonscript
        \mkern-\thinmuskip
        {:}%
        \mskip2mu
        \relax
}

\makeatletter
\newcommand{\@bbify}[1]{
  \ifcsname b#1\endcsname
  \message{WARNING: Overwriting b#1 with blackboard letter!}
  \fi
  \expandafter\edef\csname b#1\endcsname
  {\noexpand\ensuremath{\noexpand\mathbb #1}\noexpand\xspace}}
\newcommand{\@calify}[1]{
  \ifcsname c#1\endcsname
  \message{WARNING: Overwriting c#1 with calligraphic letter!}
  \fi 
  \expandafter\edef\csname c#1\endcsname
  {\noexpand\ensuremath{\noexpand\mathcal #1}\noexpand\xspace}}
\newcommand{\@bfify}[1]{
  \ifcsname bf#1\endcsname
  \message{WARNING: Overwriting c#1 with bold letter!}
  \fi
  \expandafter\edef\csname bf#1\endcsname
  {\noexpand\ensuremath{\noexpand\mathbf #1}\noexpand\xspace}}
\newcounter{@letter}\stepcounter{@letter}
\loop\@bbify{\Alph{@letter}}\@calify{\Alph{@letter}}\@bfify{\Alph{@letter}}
\ifnum\the@letter<26\stepcounter{@letter}\repeat
\makeatother

\newenvironment{tz}{\begin{center}\begin{tikzpicture}}{\end{tikzpicture}\end{center}}
\NewDocumentCommand{\punctuation}{ m m O{5pt} }{\node at ($(#1.east)-(0,#3)$) {#2};}

\newcounter{diagram}
\renewcommand{\thediagram}{\thetheorem}

\tikzset{%
node distance=1.5cm, la/.style={scale=0.8}, rr/.style={xshift=1.5cm},
space/.style={xshift=.5cm}, over/.style={auto=false,fill=white,inner sep=1.5pt, minimum size=0, outer sep=0},
    symbol/.style={%
        draw=none,
        every to/.append style={%
            edge node={node [sloped, allow upside down, auto=false]{$#1$}}},
            
    }, pro/.style={postaction={decorate,decoration={
        markings,
        mark=at position .5 with {\node at (0,0) {$\bullet$};}
      }},
      inner sep=.9ex,
      },
  n/.style={double equal sign distance, -implies}, t/.style={double distance=2.5pt, -implies, postaction={draw,-}}, 
}

\tikzstyle{d}=[double distance=.3ex]

\newcommand{\pushout}[1]{\node at ($({#1})-(10pt,-10pt)$) {$\ulcorner$};}
\newcommand{\pullback}[1]{\node at ($({#1})+(10pt,-10pt)$) {$\lrcorner$};}

\newlist{rome}{enumerate}{7}
\setlist[rome]{label=(\roman*)}
\setlist[enumerate]{label=(\arabic*)}

\newtheorem{theorem}{Theorem}[section]

\newtheorem{cor}[theorem]{Corollary}
\newtheorem{prop}[theorem]{Proposition}
\newtheorem{lemma}[theorem]{Lemma}
\declaretheorem[name=Theorem,numbered=yes]{theoremA}

\theoremstyle{definition}
\newtheorem{defn}[theorem]{Definition}

\newtheorem{ex}[theorem]{Example}
\newtheorem{notation}[theorem]{Notation}

\newtheorem*{question}{Question}

\theoremstyle{remark}
\newtheorem{rem}[theorem]{Remark}

\crefname{theorem}{Theorem}{Theorems}
\crefname{cor}{Corollary}{Corollaries}
\crefname{prop}{Proposition}{Propositions}
\crefname{lemma}{Lemma}{Lemmas}

\crefname{defn}{Definition}{Definitions}
\crefname{hyp}{Hypothesis}{Hypotheses}
\crefname{ex}{Example}{Examples}
\crefname{notation}{Notation}{Notations}
\crefname{descr}{Description}{Descriptions}
\crefname{constr}{Construction}{Constructions}

\crefname{rem}{Remark}{Remarks}

\newcommand{\set}{\mathrm{Set}}
\newcommand{\cat}{\mathrm{Cat}}
\newcommand{\sset}{\mathrm{sSet}}
\newcommand{\op}{\mathrm{op}}
\newcommand{\Dop}{\Delta^{\op}}
\newcommand{\Dsset}{\sset^{\Dop}}

\newcommand{\id}{\mathrm{id}}
\newcommand{\ho}{\mathrm{ho}}
\newcommand{\Map}{\mathrm{Map}}
\newcommand{\map}{\mathrm{map}}
\newcommand{\Sp}{\mathrm{Sp}}
\newcommand{\cosk}{\mathrm{cosk}}
\newcommand{\NI}{NI[1]}

\title{Modeling $(\infty,1)$-categories with Segal spaces}

\author{Lyne Moser}
\address{Fakultät für Mathematik, Universität Regensburg, Regensburg, Germany}
\email{lyne.moser@ur.de}

\author{Joost Nuiten}
\address{Institut de Mathématiques de Toulouse, Université Paul Sabatier, Toulouse, France}
\email{joost.nuiten@math.univ-toulouse.fr}

\begin{document}

\begin{abstract}
    In this paper, we construct a model structure for $(\infty,1)$-categories on the category of simplicial spaces, whose fibrant objects are the Segal spaces. In particular, we show that it is Quillen equivalent to the models of $(\infty,1)$-categories given by complete Segal spaces and Segal categories. We furthermore prove that this model structure has desirable properties: it is cartesian closed and left proper. As applications, we get a simple description of the inclusion of categories into $(\infty,1)$-categories and of homotopy limits of $(\infty,1)$-categories.
\end{abstract}

\maketitle

\setcounter{tocdepth}{1}
\tableofcontents

\vspace{-.5cm}

\section{Introduction}
The language of higher categories, particularly $(\infty,1)$-categories, has become increasingly prominent in modern mathematics. While an ordinary category has a set of objects and a set of morphisms between any two objects, an $(\infty,1)$-category instead provides a homotopical version of a category, whose morphisms are organized into \emph{spaces}. These spaces of maps capture homotopies between morphisms, as well as higher-dimensional homotopies between these homotopies, and composition of morphisms is then required to be associative and unital only up to coherent homotopy. As such, $(\infty,1)$-categories play a central role in various fields of mathematics that involve a notion of homotopy, such as derived algebraic geometry, $K$-theory, and topological quantum field theory. 
However, defining the structure of an $(\infty,1)$-category directly can be quite challenging because of the many coherences involved. Because of this, one classically makes use of a \emph{model} of $(\infty, 1)$-categories to develop the theory and construct examples. In this paper, by a model of $(\infty,1)$-categories, we mean a Quillen model structure on a category, in which the $(\infty,1)$-categories are represented by the cofibrant-fibrant objects.

Many of the models for $(\infty, 1)$-categories make use of simplicial objects to encode the homotopy coherent composition of morphisms.
For instance, one of the first models for $(\infty, 1)$-categories, Boardman--Vogt’s \emph{quasi-categories} \cite{BV}, describes them in terms of simplicial \emph{sets} satisfying certain horn-lifting properties. The quasi-categorical model has been developed extensively by Joyal \cite{JoyalVolumeII} and Lurie \cite{Lurie} and is nowadays widely used, but nonetheless has some drawbacks: for example, the combinatorics of quasi-categories is rather nontrivial and encodes the mapping spaces and composition operations in a somewhat indirect way.

An alternative model of $(\infty,1)$-categories, using simplicial objects valued in spaces rather than sets, has been developed by Rezk \cite{Rezk}. This model is given by the \emph{complete Segal spaces} and is obtained as a simplicial localization of the model structure on simplicial objects in spaces. Here, a Segal space is a simplicial space $X\colon \Dop \to \sset$ that satisfies the Segal condition, i.e., the Segal map $X_n \xrightarrow{\simeq} X_1 \times_{X_0} \cdots \times_{X_0} X_1$ is a weak homotopy equivalence for all $n\geq 2$. Thinking of $X_0$ and $X_1$ as spaces of objects and morphisms, the Segal condition allows one to think of $X$ as encoding the algebraic structure of an $(\infty, 1)$-category: it identifies each $X_n$ as the space of $n$ composable morphisms in $X$, so that the simplicial structure maps of $X$ describe the homotopy coherent composition of morphisms.
The additional \emph{completeness} condition asserts that the space of objects $X_0$ is weakly homotopy equivalent to the underlying $\infty$-groupoid of~$X$. This condition is somewhat different from the Segal condition: it does not really encode any algebraic structure, but instead ensures that the weak equivalences between complete Segal spaces coincide with homotopical analogues of fully faithful and essentially surjective functors, typically referred to as \emph{Dwyer--Kan equivalences}. 
In fact, weak equivalences between Segal spaces are already the Dwyer--Kan equivalences, as any Segal space can be replaced by a Dwyer--Kan equivalent complete Segal space.

To define examples of $(\infty, 1)$-categories using complete Segal spaces, it is often easier to first construct a Segal space and then replace it abstractly by a Dwyer--Kan equivalent complete Segal space. Several instances of this approach appear in the literature, including:
\begin{itemize}
\item the various types of $(\infty,1)$-categories of bordisms from \cite{LurieTFT,CS},
\item the Morita $(\infty,1)$-category of associative algebras and bimodules from \cite{Haugseng},
\item the $(\infty,1)$-category of $(\infty, 1)$-categories and correspondences from \cite{AF},
\item the $(\infty,1)$-category $BM$ with one object and monoid $M$ as endomorphisms,
\item the localization of $(\infty,1)$-categories admitting a calculus of fractions from \cite{Nuiten}.
\end{itemize}
In these examples, it is not always easy to explicitly describe the space of objects of their complete Segal space model, that is, their underlying $\infty$-groupoid. Because of this, the completeness condition can often feel inconvenient in practice, and its necessity is called into question by the fact that weak equivalences between Segal spaces are already the Dwyer--Kan equivalences. One may therefore wonder if the completeness step can be avoided entirely, motivating the following question:

\begin{question} 
Is there a model of $(\infty,1)$-categories in which the fibrant objects are the Segal spaces and the weak equivalences between fibrant objects are the Dwyer--Kan equivalences?\footnote{This question was asked 14 years ago in a MathOverflow post by Chris Schommer-Pries: \url{https://mathoverflow.net/questions/29728/a-model-category-of-segal-spaces}}
\end{question}

One possible way to address this question is by restricting to a subcategory of simplicial spaces and modeling $(\infty,1)$-categories with \emph{Segal categories}. These were originally introduced by Hirschowitz and Simpson \cite{HS} and in fact predate the notion of complete Segal spaces. They arise as the fibrant objects of a model structure, first constructed by Pellissier \cite{Pellissier} and further studied by Bergner \cite{Bergnerfibrant,Bergner} and Simpson \cite{Simpson}. In this model, the completeness condition on a Segal space is replaced by a discreteness condition, requiring the space of objects to be a set. Since this is notably not a homotopical condition, this model structure can only be defined on the full subcategory $\cP\cat(\sset)$ consisting of those simplicial spaces whose level $0$ is a set, referred to as \emph{precategory objects} in $\sset$. However, this discreteness condition is not satisfied by all of the above examples, e.g.~not by the $(\infty, 1)$-category of bordisms, and it is not always convenient to choose a set of vertices in the space of objects of a Segal space.

In this paper, we provide a positive answer to the above question. We construct a model structure on the category of simplicial spaces, where the fibrant objects are the Segal spaces, and the ``completeness'' or ``discreteness'' condition is transferred to the cofibrant objects. Specifically, the cofibrant objects are defined to be the simplicial spaces $X\colon \Dop \to \sset$ such that the space of objects $X_0$ is weakly homotopy equivalent to a set. The following result is proven in \cref{sec:existence} and is stated as \cref{thm:MS}.

\begin{theoremA} \label{thmA:MS}
        There is a cofibrantly generated model structure on the category $\Dsset$ of simplicial spaces, which we refer to as the \emph{categorical model structure} and denote by $\Dsset_\mathrm{Cat}$, in which 
    \begin{rome}
        \item the cofibrations are the monomorphisms $f\colon X\to Y$ such that there is a set $R$ and a weak equivalence $X_0\amalg R\xrightarrow{\simeq} Y_0$ in $\sset$ whose restriction to $X_0$ is $f_0$,
        \item the fibrant objects are the Segal spaces,
        \item the weak equivalences between Segal spaces are the Dwyer--Kan equivalences,
        \item the fibrations between Segal spaces are the isofibrations.
    \end{rome} 
\end{theoremA}

This model structure is very similar to the categorical model structure on categories: the cofibrations, fibrations, and weak equivalences between Segal spaces are homotopical analogues of functors that are injective on objects, isofibrations, and equivalences of categories, respectively. Furthermore, it provides, as desired, a model of $(\infty,1)$-categories, as it is Quillen equivalent to the complete Segal space model structure. In fact, it sits somewhere between Segal categories and complete Segal spaces, by the following combination of \cref{thm:idQE,LIQE,cor:rightind}. 

\begin{theoremA}
    The following functors are right Quillen equivalences
    \begin{tz}
        \node[](1) {$\Dsset_\mathrm{CSS}$}; 
        \node[right of=1,xshift=1.5cm](2) {$\Dsset_\mathrm{Cat}$}; 
        \node[right of=2,xshift=1.5cm](3) {$\cP\cat(\sset)$}; 
        \node[right of=3,xshift=1.5cm](4) {$\Dsset_\mathrm{Cat}$};

        \draw[->](1) to node[above,la]{$\id$} (2); 
        \draw[->](2) to node[above,la]{$R$} (3); 
        \draw[right hook->](3) to (4);
    \end{tz}
    where $R$ is the right adjoint of the inclusion, $\cP\cat(\sset)$ is endowed with the Segal category model structure, and $\Dsset_\mathrm{CSS}$ denotes the complete Segal space model structure on $\Dsset$. 

    Moreover, the model structure on $\cP\cat(\sset)$ is left- and right-induced from $\Dsset_\mathrm{Cat}$ along the inclusion $\cP\cat(\sset)\hookrightarrow \Dsset$. 
\end{theoremA}

Moreover, the categorical model structure retains the desirable properties of the model structure for complete Segal spaces.
The following result is a combination of \cref{thm:cartclosed,thm:leftproper}. 

\begin{theoremA}
    The model category $\Dsset_\mathrm{Cat}$ is left proper and cartesian closed. 
\end{theoremA}

These properties of the categorical model structure often allow one to deal with Segal spaces directly, without having to complete them. As a first example, recall that the nerve of a category defines a (discrete) Segal space which is generally not complete. As a consequence, the nerve does not provide an appropriate homotopical functor into the complete Segal space model structure. However, the nerve does induce a right Quillen functor from the categorical model structure on categories to the model structure on simplicial spaces from \cref{thmA:MS}, as we show in \cref{prop:nerve}. As a second example, the categorical model structure can be used to compute  pullbacks of $(\infty,1)$-categories---and more generally any limit of $(\infty,1)$-categories---modeled by non-complete Segal spaces, in a way similar to the computation of homotopy pullbacks of ordinary categories. Explicitly, we prove in \cref{prop:pullback} that the pullback of a map between Segal spaces along an isofibration is already a homotopy pullback in the complete Segal space model structure.

\subsection*{Acknowledgments}
We are grateful to João Candeias and to the referee for their careful reading of the paper and their useful comments. During the realization of this work, the first author was a member of the Collaborative Research Centre ``SFB 1085: Higher
Invariants'' funded by the DFG. The second author was supported by the CNRS, under the program
Projet Exploratoire de Premier Soutien ``Jeune chercheuse, jeune chercheur'' (PEPS
JCJC).

\section{Segal spaces, Dwyer--Kan equivalences, and isofibrations}

In this section, we provide the necessary background on Segal spaces and $(\infty,1)$-categorical notions of equivalences and fibrations between them. In \cref{sec:CSS}, we recall the notion of Segal spaces, as well as their associated model structure on simplicial spaces. Next, in \cref{sec:DKeq}, we recall the definition of Dwyer--Kan equivalences between Segal spaces and establish a useful lemma about them. Finally, in \cref{sec:isofib}, we introduce a notion of \emph{isofibrations} between Segal spaces and prove some technical results related to them.

Throughout the paper, we will make use of the following notations. We denote by $\sset$ the category of simplicial sets. 
We consider the category $\Dsset$ of simplicial objects in~$\sset$, i.e., functors $\Dop\to \sset$, which we refer to as \emph{simplicial spaces}. 

\begin{notation}
    For $m\geq 0$, we will write $F[m]\colon \Dop\to \sset$ for the representable in the \emph{categorical direction}, sending an object $[k]\in \Delta$ to the constant simplicial set at the set $\Delta([k],[m])$. We denote by $\partial F[m]$ its boundary and by $\Sp[m]$ its spine. 
    
    For $n\geq 0$, we write $\Delta[n]\colon \Dop\to \sset$ for the representable in the \emph{space direction}, given by the constant functor at the representable $\Delta[n]\in \sset$. We denote by $\partial\Delta[n]$ its boundary and by $\Lambda^k[n]$ its $k$-horn, for $0\leq k\leq n$.
\end{notation}

\begin{notation}
    Given simplicial spaces $X,Y$, we denote by $\Map(X,Y)$ the \emph{mapping space} at $X,Y$ given by the simplicial set such that, for $n\geq 0$, 
    \[ \Map(X,Y)_n\cong \Dsset(X\times \Delta[n],Y), \]
    and by $Y^X$ the \emph{internal hom} at $X,Y$ given by the simplicial space such that, for $m,n\geq 0$,
    \[ (Y^X)_{m,n}\cong \Dsset(X\times F[m]\times \Delta[n],Y). \]
\end{notation}

\subsection{Segal spaces} \label{sec:CSS}

Throughout the paper, the category $\sset$ is endowed with the Kan--Quillen model structure for Kan complexes. The category $\Dsset$ can then be endowed with the \emph{Reedy model structure}, denoted by $\Dsset_\mathrm{Reedy}$, in which the cofibrations are the monomorphisms, and the weak equivalences and (trivial) fibrations are as follows: a map $X\to Y$ in $\Dsset$ is
\begin{rome}
    \item a \emph{levelwise weak equivalence} if, for every $m\geq 0$, the induced map $X_m\to Y_m$ is a weak equivalence in $\sset$, 
    \item  a \emph{Reedy (trivial) fibration} if, for every $m\geq 0$, the $m$-th relative matching map
    \[ X_m\to Y_m\times_{\Map(\partial F[m],Y)} \Map(\partial F[m],X) \]
    is a (trivial) fibration in $\sset$. 
\end{rome}

\begin{defn}
    A simplicial space $X\colon \Dop \to \sset$ is a \emph{Segal space} if it is Reedy fibrant and, for each $m\geq 2$, the Segal map 
        \[ X_m\cong \Map(F[m],X)\to \Map(\Sp [m],X)\cong X_1\times_{X_0}\ldots \times_{X_0} X_1 \]
        is a weak equivalence in $\sset$.
\end{defn}

By localizing the Reedy model structure, one obtains a model structure on $\Dsset$ for Segal spaces. The following appears as \cite[Theorem 7.1]{Rezk}.

\begin{theorem} \label{thm:MSSegal}
    There is a cofibrantly generated, cartesian closed model structure on the category $\Dsset$, which we denote by $\Dsset_{\mathrm{Seg}}$, in which 
    \begin{rome}
        \item the cofibrations are the monomorphisms, 
        \item the fibrant objects are the Segal spaces, 
        \item the weak equivalences (resp.~fibrations) between Segal spaces are the levelwise weak equivalences (resp.~Reedy fibrations).
    \end{rome}
\end{theorem}

\subsection{Dwyer--Kan equivalences} \label{sec:DKeq}
We now recall the notion of Dwyer--Kan equivalences between Segal spaces. For this, we first review the following constructions.

\begin{defn}
    Let $X$ be a Segal space.
    \begin{rome}
        \item For $x,y\in X_{0,0}$, the \emph{mapping space} $\map_X(x,y)$ is the (homotopy) fiber at $(x,y)$ of the fibration $(d_1,d_0)\colon X_1\twoheadrightarrow X_0\times X_0$ in $\sset$.
        \item The \emph{homotopy category} of $X$ is the category $\ho X$ with object set $X_{0,0}$, and hom set at $x,y\in X_{0,0}$ given by 
        \[ \ho X(x,y)\coloneqq \pi_0\map_X(x,y), \]
        where $\pi_0\colon \sset\to \set$ is the left adjoint of the canonical inclusion $\set\hookrightarrow \sset$. Composition is induced by the Segal condition; see \cite[Proposition 5.4]{Rezk}. 
    \end{rome}
\end{defn}

\begin{defn}
    A map $f\colon X\to Y$ between Segal spaces is a \emph{Dwyer--Kan equivalence}~if 
    \begin{enumerate}
        \item it is \emph{homotopically fully faithful}, i.e., for all $x,y\in X_{0,0}$, the induced map
        \[\map_X(x,y)\to \map_Y(fx,fy) \]
        is a weak equivalence in $\sset$, and
        \item the induced functor  $\ho X\to \ho Y$ is essentially surjective on objects.
    \end{enumerate}
\end{defn}

We give an alternative description of the homotopy category in terms of a categorification. 

\begin{rem} \label{nerve}
    Recall that there is a canonical inclusion $\Delta\subseteq \cat$ into the category of categories. This yields by left Kan extension along the Yoneda embedding an adjunction 
    \[ c\colon\Dsset \rightleftarrows \cat\cocolon N \]
    which we will refer to as the \emph{categorification--nerve adjunction}. 
    In particular, given a category~$\cC$, its nerve $N\cC$ is given at $m\geq 0$ by the discrete simplicial set $(N\cC)_m\cong \cat([m],\cC)$. Since every simplicial set is Reedy fibrant as a simplicial space, the nerve $N\cC$ is a Segal space. 
\end{rem}

\begin{notation} \label{not:cosk}
    We write $(-)_0\colon \Dsset\rightleftarrows \sset\cocolon \cosk_0$ for the adjunction whose left adjoint sends a simplicial space $X$ to its underlying space $X_0$, and whose right adjoint sends $K\in \sset$ to the simplicial space $\cosk_0(K)$ given at 
    $m\geq 0$ by $\cosk_0(K)_m\cong K^{\times (m+1)}$.
\end{notation}

\begin{rem}\label{rem:reduction}
    We denote by $\cP\cat(\sset)$ the full subcategory of $\Dsset$ spanned by the simplicial spaces $X$ such that $X_0$ is a set. The canonical inclusion $\cP\cat(\sset)\hookrightarrow \Dsset$ admits a right adjoint $R\colon \Dsset\to \cP\cat(\sset)$, sending a simplicial space $X$ to the pullback in~$\Dsset$
    \begin{tz}
        \node[](1) {$RX$}; 
        \node[below of=1](2) {$X$}; 
        \node[right of=1,xshift=1cm](3) {$\mathrm{cosk}_0 (X_{0,0})$}; 
        \node[below of=3](4) {$\mathrm{cosk}_0 (X_{0})$}; 
        \pullback{1};
        \draw[right hook->](1) to (2);
        \draw[->](1) to (3);
        \draw[right hook->](3) to node[right,la]{} (4); 
        \draw[->](2) to node[below,la]{} (4);
    \end{tz}
\end{rem}

\begin{lemma} \label{lem:charhtpycat}
    For a Segal space $X$, there is a natural isomorphism of categories $c(RX)\cong \ho X$.
\end{lemma}

\begin{proof}
To provide the natural isomorphism $c(RX)\xrightarrow{\cong} \ho X$, it suffices by adjunction to construct a natural bijection between maps $RX\to N\cC$ in $\Dsset$ and functors $\ho X\to \cC$ in $\cat$. Using that $N\cC$ is $2$-coskeletal by \cite[Corollary 1.2]{JoyalVolumeII} and that $RX_1=\coprod_{x, y\in X_{0, 0}} \map_X(x, y)$, a map $RX\to N\cC$ is uniquely determined by a map $f\colon X_{0, 0}\to \mathrm{Ob}(\cC)$ together with maps of simplicial sets $f_{x, y}\colon \map_X(x, y)\to \cC(fx, fy)$ that are compatible with composition. Each space $\cC(x, y)$ is discrete, so the maps $f_{x, y}$ are uniquely determined by maps $\pi_0\map_X(x, y)\to \cC(fx, fy)$ that are compatible with composition. This is precisely the data of a functor $\ho X \to \cC$.
\end{proof}

We finally state the following characterization of homotopically fully faithfulness.

\begin{lemma} \label{charDK} 
  Let $f\colon X\to Y$ be a map between Segal spaces. The following are equivalent:
  \begin{rome}
  \item the map $f$ is homotopically fully faithful,
  \item the induced map $X_1\to Y_1\times_{Y_0^{\times 2}} X_0^{\times 2}$ is a weak equivalence in $\sset$,
  \item the map $X\to Y\times_{\cosk_0(Y_0)} \cosk_0(X_0)$ in 
  $\Dsset$ is a levelwise weak equivalence, 
 \item for all $m\geq 1$, the induced map 
 $X_m \to Y_m\times_{\Map(\partial F[m], Y)} \Map(\partial F[m], X)$ is a weak equivalence in $\sset$.
 \end{rome}
\end{lemma}

\begin{proof}
    The equivalence between (i) and (ii) follows from the fact that the following commutative square is a homotopy pullback in $\sset$
    \begin{tz}
        \node[](1) {$X_1$}; 
        \node[right of=1,xshift=1cm](2) {$Y_1$}; 
        \node[below of=1](3) {$X_0\times X_0$}; 
        \node[below of=2](4) {$Y_0\times Y_0$}; 

        \draw[->](1) to (2); 
        \draw[->](3) to (4); 
        \draw[->>](1) to node[left,la]{$\langle 0,1\rangle ^*$} (3); 
        \draw[->>](2) to node[right,la]{$\langle 0,1\rangle ^*$} (4); 
    \end{tz}
    if and only if, for all $x,y\in X_{0,0}$, the induced map on fibers 
    $\map_X(x,y)\to \map_Y(fx,fy)$ is a weak equivalence in $\sset$. Note that condition (ii) is a direct consequence of conditions (iii) and~(iv) by considering the induced maps in simplicial degree $m=1$. To see that (ii) implies (iii) and (iv), it suffices to consider the induced maps in simplicial degrees $m\geq 2$. To this end, consider the following diagram in $\sset$.
        \begin{tz}
        \node[](1) {$X_m$}; 
        \node[right of=1,xshift=2cm](2) {$\Map(\partial F[m],X)$}; 
        \node[right of=2,xshift=2.7cm](3) {$\Map(\Sp[m],X)$}; 
        \node[right of=3, xshift=2cm](4) {$X_0^{\times (m+1)}$};

        \draw[->>](1) to (2); 
        \draw[->>](2) to (3);
        \draw[->>](3) to (4);

        \node[below of=1](1') {$Y_m$}; 
        \node[below of=2](2') {$\Map(\partial F[m],Y)$}; 
        \node[below of=3](3') {$\Map(\Sp[m],Y)$}; 
        \node[below of=4](4') {$Y_0^{\times (m+1)}$};

        \draw[->>](1') to (2'); 
        \draw[->>](2') to (3');
        \draw[->>](3') to (4');

        \draw[->](1) to (1');
        \draw[->](2) to (2');
        \draw[->](3) to (3');
        \draw[->](4) to (4');
    \end{tz}
    The horizontal maps are fibrations in $\sset$ since $X$ and $Y$ are Reedy fibrant. Consequently, condition (iii) is equivalent to the total rectangle being a homotopy pullback and condition~(iv) is equivalent to the left square being a homotopy pullback. We will prove these are homotopy pullbacks by induction on~$m$, the initial case $m=1$ being precisely (ii). The composite of the two leftmost squares is a homotopy pullback, since the top and bottom maps are both weak equivalences in $\sset$ by the Segal conditions on $X$ and $Y$. Furthermore, the right-hand square and the composition of the two rightmost squares are both homotopy pullback squares. Indeed, this follows from the fact that the monomorphisms
    \[ 
    \{0, \dots, m\}\hookrightarrow \Sp[m] \quad\text{and}\quad \{0, \dots, m\}\hookrightarrow \partial F[m]
    \]
    are both iterated pushouts of maps $\partial F[k]\hookrightarrow F[k]$ with $1\leq k<m$, and, by the inductive hypothesis, each such pushout of $\partial F[k]\hookrightarrow F[k]$ induces a homotopy pullback square. By the cancellation property of homotopy pullbacks, we get that the left-hand square and the total rectangle are homotopy pullbacks, as desired.
\end{proof}
\begin{ex} \label{levelwiseDK}
Every weak equivalence $f\colon X\to Y$ between Segal spaces in $\Dsset_\mathrm{Seg}$ is a Dwyer--Kan equivalence. Indeed, we can write $f$ as the composition of a Reedy trivial fibration $p\colon X'\to Y$ and a section of a Reedy trivial fibration $i\colon X\to X'$. Since $p$ is surjective on objects, it is a Dwyer--Kan equivalence by \cref{charDK}. The map $i$ is a Dwyer--Kan equivalence by 2-out-of-3.
\end{ex}

\subsection{Isofibrations} \label{sec:isofib}

Finally, let us introduce the following version of isofibrations between Segal spaces. For this we first recall that classically, an \emph{isofibration} is a functor with the right lifting property against either inclusion $[0]\hookrightarrow I[1]$, where $I[1]$ denotes the free-living isomorphism.

\begin{defn}
Let $X,Y$ be Segal spaces. A map $f\colon X\to Y$ is said to be an \emph{isofibration} if it is a Reedy fibration and the induced functor $\ho X\to \ho Y$ is an isofibration of categories.
\end{defn}

We aim to characterize these isofibrations through a lifting property. We will deduce this from an analogous result for quasi-categories due to Joyal. For this, recall that the homotopy category of a quasi-category $X$ can be modeled by $cX$, where $c\colon \set^{\Dop} \to \cat$ is the left adjoint of the usual nerve $N\colon \cat\to \set^{\Dop}$. We then have the following lemma relating the homotopy category of a Segal space and that of its underlying quasi-category. 

\begin{lemma} \label{lem:undqcat}
    Let $X,Y$ be Segal spaces and $X\to Y$ be a Reedy fibration. Then:
    \begin{rome}
        \item the underlying simplicial set $X_{-,0}$ is a quasi-category, 
        \item the induced map $X_{-,0}\to Y_{-,0}$ is an inner fibration between quasi-categories,
        \item the homotopy category of the quasi-category $X_{-,0}$ is naturally isomorphic to $\ho X$.
    \end{rome}
\end{lemma}

\begin{proof}
    Assertions (i) and (ii) follow from the fact that each inner horn inclusion $L^k[m]\hookrightarrow F[m]$ for $0<k<m$, where $L^k[m]$ denotes the $k$-th horn, in the categorical direction is a trivial cofibration in $\Dsset_\mathrm{Seg}$ by \cite[Lemma 3.5]{JT}. 
    
    For (iii), using \cref{lem:charhtpycat}, it suffices to show that the inclusion $X_{-, 0}\to RX$ induces an isomorphism on categorifications. For this, recall from \cite[Proposition 1.11]{JoyalVolumeII} that the category $c(X_{-,0})$ has set of objects $X_{0, 0}$, and hom set at $x, y\in X_{0, 0}$ given by the quotient of $\map_X(x, y)_0$ by the following equivalence relation: $\alpha\sim \beta$ if there exists $h\in X_{2, 0}$ such that $d_2h=\alpha$, $d_1h=\beta$, and $d_0h=s_0y$. Unraveling the definitions, the functor $c(X_{-, 0})\to c(RX)$ is now given by the identity on objects and, on hom sets, by the canonical map 
    \[ \map_X(x, y)_0/_\sim\to \pi_0\map_X(x, y) \]
    sending $\alpha$ to its path component. In particular, if $\alpha\sim \beta$ in $\map_X(x, y)_0$, then $\alpha,\beta$ lie in the same path component of $\map_X(x, y)$.

To see that $c(X_{-, 0})\to c(RX)$ is an isomorphism, it remains to show that $\alpha\sim \beta$ as soon as $\alpha, \beta \in \map_X(x, y)_0$ lie in the same path component of $\map_X(x,y)$. Let $h'\colon F[1]\times \Delta[1]\to RX\hookrightarrow X$ be a homotopy in between $\alpha,\beta\in \map_X(x,y)_0$, and consider the extension problem below left.
\begin{tz}
        \node[](1) {$\partial F[2]\times \Delta[1]\amalg_{\partial F[2]\times\{0\}} F[2]\times\{0\}$}; 
        \node[right of=1,xshift=3cm](2) {$X$}; 
        \node[below of=1](3) {$F[2]\times \Delta[1]$}; 

        \draw[->](1) to (2); 
        \draw[right hook->](1) to (3); 
        \draw[->, dashed]($(3.east)$) to node[below,la]{$H$} ($(2.south west)$);

        \node[right of=2,xshift=2cm,yshift=.25cm](1) {$x$};
        \node[below right of=1,yshift=.3cm,xshift=.5cm](1') {$y$}; 
        \node[above right of=1',yshift=-.3cm,xshift=.5cm](1'') {$y$};
        \node[below of=1](2) {$x$};
        \node[below of=1'](2') {$y$}; 
        \node[below of=1''](2'') {$y$};

        \draw[d](1) to (2);
        \draw[d](1'') to (2'');
        \draw[->](1) to node[below,la]{$\alpha$} (1');
        \draw[->](1) to node(a)[above,la]{$\alpha$} (1'');
        \draw[->](2) to node[below,la]{$\beta$} (2');
        \draw[->](2) to node[above,la,xshift=10pt]{$\alpha$} (2'');
        \draw[d](1') to (1'');
        \draw[d](2') to (2'');
        \node[la] at ($(1)!0.5!(2')$) {$h'$};
        \draw[-,line width=2mm,white](1') to (2');
        \draw[d](1') to (2');
\end{tz}
Here the top map is given by the picture as above right, with the top, back, and right-hand faces degenerate and the left-hand face given by $h'$. Since $X$ is Reedy fibrant, there exists a dashed extension $H$, whose restriction to $F[2]\times \{1\}$ provides an element $h\in X_{2,0}$ that exhibits $\alpha\sim \beta$.
\end{proof}

\begin{prop}\label{lem:isofib}
Let $f\colon X\to Y$ be a Reedy fibration between Segal spaces. Then the following are equivalent:
\begin{rome}
\item the map $f$ is an isofibration,
\item the map $f$ has the right lifting property against either inclusion $F[0]\hookrightarrow \NI$, where $\NI$ denotes the nerve of the free-living isomorphism.
\end{rome}
\end{prop}

\begin{proof}
The map $f\colon X\to Y$ has the right lifting property against $F[0]\hookrightarrow\NI$ if and only if the induced map $f_{-,0}\colon X_{-,0}\to Y_{-,0}$ has the right lifting property against $F[0]\hookrightarrow \NI$ (viewed as a map of simplicial sets). The map $f_{-,0}$ is an inner fibration between quasi-categories by \cref{lem:undqcat} (ii), so that \cite[Proposition 2.4]{Joyal} implies that the map $f_{-,0}$ has the said lifting property if and only if the induced functor $c(X_{-,0})\to c(Y_{-,0})$ between homotopy categories is an isofibration. By \cref{lem:undqcat}~(iii), we conclude that this is the case if and only if $f$ is an isofibration. 
\end{proof}

Finally, we prove a technical result about isofibrations that will be useful later on. 
\begin{defn}
    Let $X$ be a Segal space.
    \begin{rome}
      \item  The \emph{space of homotopy equivalences} of $X$ is the subspace $X_\mathrm{hoeq}\subseteq X_1$ consisting of those maps $F[1]\to X$ whose image in the homotopy category $\ho X$ is an isomorphism.
      \item The \emph{underlying Segal groupoid} of $X$ is the sub-simplicial space $X^\simeq\subseteq X$ given at $m\geq 0$ by the simplicial set
      \[ \textstyle (X^{\simeq})_m\coloneqq X_m\times_{\prod_{[1]\to [m]} X_1} (\prod_{[1]\to [m]} X_\mathrm{hoeq}). \]
    \end{rome}
\end{defn}

\begin{lemma} \label{lemma:undgpd}
    Let $X,Y$ be Segal spaces and $X\to Y$ be an isofibration in $\Dsset$. Then:
    \begin{rome}
    \item every map $\NI\to X$ factors as $\NI\to X^{\simeq}\hookrightarrow X$,
    \item every map $\NI\times \NI\to X$ factors as $\NI\times \NI\to X^{\simeq}\hookrightarrow X$,
    \item the underlying simplicial set $(X^\simeq)_{-,0}$ is a Kan complex, 
    \item the induced map $(X^\simeq)_{-,0}\to (Y^\simeq)_{-,0}$ is a fibration in $\sset$.
    \end{rome}
\end{lemma}

\begin{proof}
    To see (i), it suffices to show that the $1$-simplices of $\NI$ map to $1$-simplices of $X$ whose image in the homotopy category $\ho X$ is an isomorphism. But this follows from the fact that $\NI\to X$ induces a functor $I[1]\cong \ho(NI[1])\to \ho X$. The proof of (ii) works similarly, using that $I[1]\times I[1]\cong \ho(\NI\times \NI)$.
    
    For (iii), we show that we have the following pullback squares in $\sset$. 
    \begin{tz}
        \node[](1) {$(X^\simeq)_m$}; 
        \node[right of=1,xshift=2cm](2) {$\Map(\partial F[m],X^\simeq)$}; 
        \node[below of=1](3) {$X_m$}; 
        \node[below of=2](4) {$\Map(\partial F[m],X)$}; 
        \pullback{1};

        \draw[->>](1) to (2); 
        \draw[->>](3) to (4); 
        \draw[right hook->](1) to (3); 
        \draw[right hook->](2) to (4); 

        \node[right of=2,xshift=3.5cm](2') {$X_\mathrm{hoeq}\times_{X_0}\ldots\times_{X_0} X_\mathrm{hoeq}$}; 
        \node[below of=2'](4') {$X_1\times_{X_0}\ldots \times_{X_0} X_1$}; 
        \pullback{2};

        \draw[->>](2) to (2'); 
        \draw[->>](4) to (4'); 
        \draw[right hook->](2') to (4');
    \end{tz}
    To see that the outer and right-hand square are pullbacks, note that the pullback of the given cospan is the space of maps $(\partial)F[m]\to X$ whose restriction to $\Sp[m]$ classifies $m$ composable morphisms in $X$ that all induce isomorphisms in $\ho X$. Any composition of these morphisms defines an isomorphism in $\ho X$ as well, so that any composite $F[1]\to (\partial)F[m]\to X$ defines an element in $X_\mathrm{hoeq}$. This shows that the pullback is isomorphic to $(X^\simeq)_m$ or $\Map(\partial F[m],X^\simeq)$. The left-hand square is then also a pullback square by cancellation. Using this, the fact that $X$ is a Segal space implies that $X^\simeq$ is also a Segal space. By \cref{lem:undqcat} (i), we get that $(X^\simeq)_{-,0}$ is a quasi-category, whose homotopy category coincides with $\ho(X^\simeq)\subseteq \ho X$ by \cref{lem:undqcat} (iii). Every morphism in the quasi-category $(X^\simeq)_{-,0}$ is therefore invertible, so that $(X^\simeq)_{-,0}$ is a Kan complex by \cite[Corollary 1.4]{Joyal}.

    For (iv), using (iii), it suffices to show that $(X^\simeq)_{-,0}\to (Y^\simeq)_{-,0}$ is an inner fibration which lifts against $F[0]\hookrightarrow \NI$. Using the above left pullback square, one can show that the $m$-th relative matching map of $X^\simeq\to Y^\simeq$ is a pullback of the $m$-th relative matching map of $X\to Y$. The fact that $X\to Y$ is a Reedy fibration then implies that ${X^\simeq\to Y^\simeq}$ is also a Reedy fibration. By \cref{lem:undqcat} (ii), we get that $(X^\simeq)_{-,0}\to (Y^\simeq)_{-,0}$ is an inner fibration. Using \cite[Proposition 2.4]{Joyal} and the isomorphisms $\ho((X^\simeq)_{-,0})\cong \ho (X^\simeq)$ and $\ho((Y^\simeq)_{-,0})\cong \ho (Y^\simeq)$ from \cref{lem:undqcat} (iii), it suffices to show that $X^\simeq\to Y^\simeq$ has the right lifting property against $F[0]\hookrightarrow \NI$. By \cref{lem:isofib} and (i), this follows from the fact that $X\to Y$ is an isofibration.
\end{proof}

\begin{lemma} \label{technicalisofib}
    Let $X\to Y$ be an isofibration between Segal spaces. Then the induced maps
    \[ p\colon X^{\NI}\to Y^{\NI}\times_{Y\times Y} (X\times X) \quad \text{and}\quad q\colon X^{\NI}\to Y^{\NI}\times_Y X \]
    are isofibrations. 
\end{lemma}

\begin{proof}
    The maps $F[0]\amalg F[0]\hookrightarrow \NI$ and $F[0]\hookrightarrow \NI$ are cofibrations and $X\to Y$ is a fibration between fibrant objects in $\Dsset_{\mathrm{Seg}}$. The cartesian closedness of $\Dsset_{\mathrm{Seg}}$ then implies that $p$ and $q$ are fibrations in $\Dsset_{\mathrm{Seg}}$ between fibrant objects, i.e., Reedy fibrations between Segal spaces.
    It thus remains to show that they have the right lifting property against $F[0]\hookrightarrow \NI$. 
        
    For the map $p$, this is equivalent to showing that $X\to Y$ has the right lifting property against 
    \[ \NI\amalg_{ F[0]\amalg F[0]} (\NI\amalg \NI)\hookrightarrow \NI\times \NI. \]
    Using \cref{lemma:undgpd} (i) and (ii), this is equivalent to showing that $X^\simeq\to Y^\simeq$ has the right lifting property against the above map. This follows from the fact that $(X^\simeq)_{-,0}\to (Y^\simeq)_{-,0}$ is fibration in $\sset$ by \cref{lemma:undgpd}~(iv).

    The right lifting property of the map $q$ against $F[0]\hookrightarrow \NI$ is equivalent to the right lifting property of $X\to Y$ against the composite map
     \[ \NI\amalg_{F[0]} \NI\hookrightarrow\NI\amalg_{ F[0]\amalg F[0]} (\NI\amalg \NI)\hookrightarrow \NI\times \NI. \]
     The map $X\to Y$ has the right lifting property against the first map, as this is a pushout of the map $F[0]\hookrightarrow \NI$ and $X\to Y$ is an isofibration, and against the second map by the above. 
\end{proof}

\section{Construction of the model structure} \label{sec:existence}

In this section, we aim to prove the existence of the following model structure. 

\begin{theorem} \label{thm:MS}
    There is a cofibrantly generated model structure on $\Dsset$, which we refer to as the \emph{categorical model structure} and denote by $\Dsset_\mathrm{Cat}$, in which 
    \begin{rome}
        \item the cofibrations are the monomorphisms $f\colon X\hookrightarrow Y$ such that there is a set $R$ and a weak equivalence $X_0\amalg R\xrightarrow{\simeq} Y_0$ in $\sset$ whose restriction to $X_0$ is $f_0$,
        \item the fibrant objects are the Segal spaces,
    \item the weak equivalences between Segal spaces are the Dwyer--Kan equivalences,
    \item the fibrations between Segal spaces are the isofibrations.
    \end{rome}
\end{theorem}

For this, we will use \cite[Theorem 2.8]{GMSV}. To recall the statement, we first introduce the following terminology for a locally presentable category $\cC$ and a set $\cI$ of morphisms.

\begin{defn}
     We say that a map $X\to Y$ in $\cC$ is
    \begin{rome}
        \item an \emph{$\cI$-fibration} if it has the right lifting property against every morphism in $\cI$; we denote by $\cI\text{-}\mathrm{fib}$ the class of all $\cI$-fibrations, 
        \item an \emph{$\cI$-cofibration} if it has the left lifting property against every $\cI$-fibration; we denote by $\cI\text{-}\mathrm{cof}$ the class of all $\cI$-cofibrations. 
    \end{rome}
    By the small object argument, the pair $(\cI\text{-}\mathrm{cof},\cI\text{-}\mathrm{fib})$ forms a weak factorization system on $\cC$. 
\end{defn}

Given a weak factorization system as above, we can introduce the following notions of fibrant objects and fibrant replacements. 

\begin{defn}
We introduce the following terminology. 
    \begin{rome}
        \item An object $X\in \cC$ is \emph{$\cI$-fibrant} if the unique morphism $X\to *$ to the terminal object is an $\cI$-fibration. 
        \item An \emph{$\cI$-fibrant replacement} of an object $X\in \cC$ is an $\cI$-fibrant object $\widetilde{X}$ together with an $\cI$-cofibration $X\to \widetilde{X}$. 
        \item An \emph{$\cI$-fibrant replacement} of a morphism $X\to Y$ is a morphism $\widetilde{X}\to \widetilde{Y}$ between $\cI$-fibrant objects fitting into a commutative square 
\begin{tz}
        \node[](1) {$X$}; 
        \node[right of=1](2) {$Y$}; 
        \node[below of=1](3) {$\widetilde{X}$}; 
        \node[below of=2](4) {$\widetilde{Y}$}; 

        \draw[->](1) to node[above,la]{} (2); 
        \draw[->](3) to node[below,la]{} (4); 
        \draw[->](1) to node[left,la]{$\cI\text{-}\mathrm{cof}\ni$} (3); 
        \draw[->](2) to node[right,la]{$\in\cI\text{-}\mathrm{cof}$} (4); 
    \end{tz}
    \end{rome}
     Since $(\cI\text{-}\mathrm{cof},\cI\text{-}\mathrm{fib})$ is a weak factorization system, such $\cI$-fibrant replacements always exist.
\end{defn}

The following theorem is a combination of \cite[Theorem 2.8 and Proposition 2.21]{GMSV}. 

    \begin{theorem}\label{thm:main}
Let $\cC$ be a locally presentable category, and let $\cI$ and $\cJ$ be sets of morphisms in $\cC$ such that $\cJ\subseteq \cI\text{-}\mathrm{cof}$. Suppose in addition that we have a class $\cW_f$ of morphisms in $\cC$ between $\cJ$-fibrant objects. Suppose that the following conditions are satisfied:
\begin{enumerate}[label=(\Roman*)]
    \item\label{2of6Wf} $\cW_f$ satisfies the 2-out-of-6 property, 
    \item \label{accessibility} there exists a class $\overline{\cW}$ of morphisms in $\cC$ such that $\cW_f$ is the restriction of $\overline{\cW}$ to the morphisms between $\cJ$-fibrant objects and such that $\overline{\cW}$ considered as a full subcategory of $\cC^{[1]}$ is accessible, 
    \item\label{path} for every $\cJ$-fibrant object $X$, there is a factorization of the diagonal morphism 
    \[ X\xrightarrow{w} \mathrm{Path} X\xrightarrow{p} X\times X \]
    such that $w\in \cW_f$ and $p\in \cJ\text{-}\mathrm{fib}$,
    \item\label{fibwe} $\cJ\text{-}\mathrm{fib}\cap\cW_f\subseteq\cI\text{-}\mathrm{fib}$,
    \item\label{ax:trivfib} $\cI\text{-}\mathrm{fib}\subseteq\cW$, where $\cW$ is the class of morphisms in $\cC$ which admit a $\cJ$-fibrant replacement that is in $\cW_f$.
\end{enumerate}
Then there is a cofibrantly generated model structure on $\cC$, in which 
\begin{rome}
    \item the cofibrations are the $\cI$-cofibrations, 
    \item the fibrant objects are the $\cJ$-fibrant objects,
    \item the weak equivalences between fibrant objects are the morphisms in $\cW_f$, 
    \item the fibrations between fibrant objects are the $\cJ$-fibrations.
\end{rome}
\end{theorem}

In the remainder of this section, we prove \cref{thm:MS} using \cref{thm:main}. First, in \cref{sec:cof,sec:fib}, we describe the generating sets $\cI$ and $\cJ$, respectively, and study the weak factorization systems they generate. In particular, we show that $\cI$-cofibrations and $\cJ$-fibrations between $\cJ$-fibrant objects align with the descriptions provided in (i), (ii), and (iv) of \cref{thm:MS}. Next, in \cref{sec:we}, we define $\cW_f$ as the class of Dwyer–Kan equivalences between Segal spaces, obtaining the description given in (iii) of \cref{thm:MS}. We also verify that, with this definition, Conditions \ref{2of6Wf} and \ref{accessibility} of \cref{thm:main} are satisfied. Finally, in \cref{sec:pathobj,sec:fibwearetrivfib,sec:trivfibarewe}, we establish that Conditions \ref{path}, \ref{fibwe}, and~\ref{ax:trivfib} of \cref{thm:main} hold in our setting, thereby completing the proof of \cref{thm:MS}.

\subsection{Cofibrations and trivial fibrations} \label{sec:cof}

Let us start by introducing the set $\cI$ of generating cofibrations.

\begin{notation}\label{cI}
    Let $\cI$ denote the set of maps in $\Dsset$ consisting of
    \begin{rome}
        \item for all $m\geq 1$ and $n\geq 0$ and $m=n=0$, 
        the monomorphism
        \[ \partial F[m]\times \Delta[n]\amalg_{\partial F[m]\times \partial\Delta[n]} F[m]\times \partial\Delta [n]\hookrightarrow F[m]\times \Delta[n], \]
        \item for all $n\geq 0$ and $0\leq k\leq n$, the monomorphism $\Lambda^k[n]\hookrightarrow \Delta[n]$.
    \end{rome}
\end{notation}

The following characterization of $\cI$-fibrations is straightforward from unpacking their lifting properties against maps in $\cI$.

\begin{prop} \label{Ifibration}
    A map $X\to Y$ in $\Dsset$ is an $\cI$-fibration if and only if the following conditions hold:
    \begin{enumerate}
        \item for all $m\geq 1$, the induced map $X_m \to Y_m\times_{\Map(\partial F[m],Y)} \Map(\partial F[m],X)$ is a trivial fibration in $\sset$, 
        \item the induced map $X_0\to Y_0$ is a fibration in $\sset$, 
        \item the induced map $X_{0,0}\to Y_{0,0}$ is surjective. 
    \end{enumerate}
\end{prop}

In addition, we show that the $\cI$-cofibrations coincide with the cofibrations described in (i) of \cref{thm:MS}.

\begin{lemma} \label{lem:charcofib0}
    The following conditions are equivalent for a monomorphism $f\colon X\hookrightarrow Y$ in $\Dsset$:
    \begin{rome}
        \item there is a set $R$ and a weak equivalence $X_0\amalg R\xrightarrow{\simeq} Y_0$ whose restriction to $X_0$ is~$f_0$,
        \item for each connected component $S\subseteq Y_0$, either the induced map $f_0^{-1}(S)\to S$ is a weak equivalence in $\sset$ or we have $f_0^{-1}(S)=\emptyset$ and $S\simeq \Delta[0]$.
    \end{rome}
\end{lemma}

\begin{proof}
  We see that (i) immediately implies (ii), by verifying the desired conditions on the equivalent map $X_0\hookrightarrow X_0\amalg R$ to $f_0$. For the converse, we take $R$ to be the set consisting of one point in each connected component $S\subseteq Y_0$ such that $f_0^{-1}(S)=\emptyset$. Then the map $X_0\amalg R\to Y_0$ is a weak equivalence in $\sset$, as it induces a weak equivalence over each path component of the target.
\end{proof}

\begin{prop} \label{Icof}
    A map $f\colon X\to Y$ in $\Dsset$ is an $\cI$-cofibration if and only if the following conditions hold:
    \begin{enumerate}
        \item it is a monomorphism, 
        \item there exist a set $R$ and a weak equivalence $X_0\amalg R\xrightarrow{\simeq} Y_0$ in $\sset$ whose restriction to $X_0$ is $f_0$.
    \end{enumerate}
    In particular, a simplicial space $X$ is $\cI$-cofibrant if and only if $X_0$ is weakly equivalent to a set.
\end{prop}

\begin{proof}
    Consider the class $\cA$ of all monomorphisms satisfying (2). We show that $\cA$ is equal to the class of all $\cI$-cofibrations.
    
    First note that $\cA$ contains all the monomorphisms in $\cI$. We show that $\cA$ is saturated. The fact that $\cA$ is closed under pushouts follows from the fact that taking pushout along a map in $\sset$ induces a left Quillen functor between slices. To show that $\cA$ is closed under transfinite compositions and retracts, we use the characterization~(ii) from \cref{lem:charcofib0}. If $f$ arises as a transfinite composition of maps in $\cA$, then each $f_0^{-1}(S)\to S$ is either a transfinite composition of trivial cofibrations, hence a weak equivalence, or $f_0^{-1}(S)=\emptyset$ and $S$ is a filtered colimit of weakly contractible complexes, hence weakly contractible. It follows that $f\in \cA$ as well. Finally, suppose that we have a retract diagram in $\Dsset$ as below left, where $f\in \cA$. Given a connected component $S \subseteq W_0$, we get an induced retract diagram in $\sset$ as below right.
    \begin{tz}
         \node[](1) {$Z$}; 
        \node[below of=1](2) {$W$}; 
        \node[right of=1](3) {$X$}; 
        \node[below of=3](4) {$Y$}; 
        \node[right of=3](5) {$Z$}; 
        \node[below of=5](6) {$W$}; 
        \draw[->](1) to (3);
        \draw[->](2) to node[below, la]{$i$} (4); 
        \draw[->](3) to (5);
        \draw[->](4) to node[below,la]{$r$} (6); 
        \draw[->](1) to node[left,la]{$g$} (2);
        \draw[->](3) to node[right,la]{$f$} (4);
        \draw[->](5) to node[right,la]{$g$} (6);

         \node[right of=5,xshift=2cm](1) {$g_0^{-1}(S)$}; 
        \node[below of=1](2) {$S$}; 
        \node[right of=1,xshift=1cm](3) {$f_0^{-1}(T)$}; 
        \node[below of=3](4) {$T$}; 
        \node[right of=3,xshift=1cm](5) {$g_0^{-1}(S)$}; 
        \node[below of=5](6) {$S$}; 
        \draw[->](1) to (3);
        \draw[->](2) to (4); 
        \draw[->](3) to (5);
        \draw[->](4) to (6); 
        \draw[->](1) to node[left,la]{$g_0$} (2);
        \draw[->](3) to node[right,la]{$f_0$} (4);
        \draw[->](5) to node[right,la]{$g_0$} (6);
    \end{tz}
    Here we write $T\subseteq r_0^{-1}(S)\subseteq Y_0$ for the connected component of $Y_0$ containing the image $i_0(S)$ of the connected component $S\subseteq W_0$. By assumption, we have either that $f_0^{-1}(T)\to T$ is a weak equivalence in $\sset$ or that $f_0^{-1}(T)=\emptyset$ and $T\simeq \Delta[0]$. The result then follows from the fact that both of these conditions are closed under retracts. As the class of $\cI$-cofibrations is the smallest saturated class containing~$\cI$, we get that $\cI\text{-}\mathrm{cof}\subseteq \cA$.
    
Conversely, let $f\colon X\hookrightarrow Y$ be a monomorphism in $\cA$. Let $R$ be the set of path components of $Y_0$ that are not in the image of $f_0$, and choose a map $R\to Y_0$ selecting exactly one vertex in each such path component. The map $f$ then factors as an $\cI$-cofibration $X\hookrightarrow X\amalg R$ followed by a monomorphism $f'\colon X\amalg R\hookrightarrow Y$ such that $f'_0$ is a weak equivalence; in particular, $f'$ is also in~$\cA$, and it suffices to verify that it is contained in $\cI\text{-}\mathrm{cof}$. 

We now apply the small object argument with respect to the set of maps $\cI'=\cI\setminus\{\emptyset\to F[0]\}$, in order to factor the map $f'$ into an $\cI'$-cofibration $i\colon X\amalg R\hookrightarrow Z$ followed by an $\cI'$-fibration $p\colon Z\to Y$. We claim that $p$ is a Reedy trivial fibration. Using this and the fact that $f'$ is a monomorphism, we get that $f'$ is a retract of $i$. It follows that $f'$ is an $\cI'$-cofibration and hence an $\cI$-cofibration.

To see that $p$ is a Reedy trivial fibration, note that the right lifting property against the maps in $\cI'$ implies that $p$ satisfies properties (1) and (2) of \cref{Ifibration}. It therefore remains to check that $p_0\colon Z_0\to Y_0$ is a weak equivalence in $\sset$, or equivalently, by the 2-out-of-3 property, that $i_0\colon X_0\amalg R\to Z_0$ is a weak equivalence. This follows from that fact that every map in $\cI'$ induces a trivial cofibration in $\sset$ in degree $0$, so that $i_0\colon X_0\amalg R\to Z_0$ is a trivial cofibration in $\sset$.
\end{proof}

\subsection{Anodyne extensions and fibrations between fibrant objects} \label{sec:fib}

We now introduce the set $\cJ$ of generating anodyne extensions.

\begin{notation} \label{notn:J}
    Let $\cJ$ denote the set of maps in $\Dsset$ containing
    \begin{rome}
        \item for all $m\geq 0$, $n\geq 1$, and $0\leq k\leq n$, the monomorphism 
        \[ \partial F[m]\times \Delta[n]\amalg_{\partial F[m]\times \Lambda^k[n]} F[m]\times \Lambda^k[n]\hookrightarrow F[m]\times \Delta[n], \]
        \item for all $m\geq 2$ and $n\geq 0$, the monomorphism
        \[ \Sp[m]\times \Delta[n]\amalg_{\Sp[m]\times \partial\Delta[n]} F[m]\times \partial\Delta [n]\hookrightarrow F[m]\times \Delta[n], \]
        \item either inclusion $F[0]\hookrightarrow \NI$.
    \end{rome}
    We write $\cJ_0\subseteq \cJ$ for the subset of monomorphisms of the form (i) and (ii).
\end{notation}
\begin{rem} \label{rem:J0fib}
Every $\cJ_0$-cofibration is in particular a trivial cofibration in $\Dsset_{\mathrm{Seg}}$. In fact, the $\cJ_0$-cofibrations are precisely the maps in $\Dsset_{\mathrm{Seg}}$ that determine the fibrant objects and the fibrations between fibrant objects. In particular, a $\cJ_0$-fibrant object is a Segal space and a $\cJ_0$-fibration between $\cJ_0$-fibrant objects is a Reedy fibration. 
\end{rem}

We now aim to show that $\cJ$-fibrant objects and $\cJ$-fibrations between them align with the descriptions given in (ii) and (iv) of \cref{thm:MS}. For this, we first prove that $\cJ$-fibrant replacement can be computed as $\cJ_0$-fibrant replacements. 

\begin{lemma}\label{lem:cJ0}
Let $X$ be a simplicial space. Then the following assertions hold:
\begin{rome}
    \item the simplicial space $X$ is $\cJ$-fibrant if and only if it is $\cJ_0$-fibrant,
    \item there is a $\cJ$-fibrant replacement $j_X\colon X\to \widetilde{X}$, where $j_X$ is a $\cJ_0$-cofibration.
\end{rome}
\end{lemma}

\begin{proof}
Since the map $F[0]\hookrightarrow \NI$ has a retraction, every simplicial space has the extension property against $F[0]\hookrightarrow \NI$. This shows (i). For (ii), we can apply the small object argument with respect to $\cJ_0$ to provide the desired $\cJ$-fibrant replacement.
\end{proof}

\begin{prop} \label{Jfibration}
A map $f\colon X\to Y$ in $\Dsset$ between $\cJ$-fibrant objects is a $\cJ$-fibration if and only if the following conditions hold:
\begin{enumerate}
    \item the simplicial spaces $X,Y$ are Segal spaces,
    \item the map $f$ is an isofibration. 
\end{enumerate}
In particular, a simplicial space $X$ is $\cJ$-fibrant if and only if it is a Segal space.
\end{prop}

\begin{proof}
As a $\cJ$-fibration is by definition a $\cJ_0$-fibration with the right lifting property against $F[0]\hookrightarrow \NI$, the result follows from  \cref{lem:isofib}, \cref{lem:cJ0} (i), and \cref{rem:J0fib}.
\end{proof}

We now prove a useful property of $\cJ_0$-cofibrations, which can be applied in particular to the $\cJ$-fibrant replacements from \cref{lem:cJ0}.

\begin{lemma} \label{J0coflevel01}
    If $X\to Y$ in $\Dsset$ is a $\cJ_0$-cofibration, then the induced map $X_0\to Y_0$ is a trivial cofibration in $\sset$.
\end{lemma}

\begin{proof}
    Consider the class $\cA$ of all monomorphisms $X\hookrightarrow Y$ in $\Dsset$ such that the induced map $X_0\hookrightarrow Y_0$ is a trivial cofibration in $\sset$. We show that every $\cJ_0$-cofibration is in $\cA$. Note that $\cJ_0\subseteq \cA$ and that $\cA$ is saturated since colimits and retracts in $\Dsset$ can be computed levelwise in $\sset$ and the class of trivial cofibrations in $\sset$ is saturated. As the class of $\cJ_0$-cofibrations is the smallest saturated class containing~$\cJ_0$, we get that $\cJ_0\text{-}\mathrm{cof}\subseteq \cA$, as desired. 
\end{proof}

\subsection{Weak equivalences} \label{sec:we}

Finally, we introduce the class of weak equivalences. To get the desired description of the weak equivalences between fibrant objects stated in (iii) of \cref{thm:MS}, we define $\cW_f$ to be the class of Dwyer--Kan equivalences between Segal spaces.

With this definition, we now show that Conditions \ref{2of6Wf} and \ref{accessibility} of \cref{thm:main} hold.

\begin{prop}
    The class $\cW_f$ of Dwyer--Kan equivalences between Segal spaces satisfies the $2$-out-of-$6$ property. 
\end{prop}

\begin{proof}
    This follows from the definition of a Dwyer--Kan equivalence, using that equivalences of categories and weak equivalences in $\sset$ satisfy the 2-out-of-6 property. 
\end{proof}

\begin{prop}
    There is a class $\overline{\cW}$ of maps in $\Dsset$ such that $\cW_f$ is the restriction of $\overline{\cW}$ to the maps between Segal spaces and $\overline{\cW}$ considered as a full subcategory of $(\Dsset)^{[1]}$ is accessible.
\end{prop}

\begin{proof}
    Consider the functor $\Phi\colon (\Dsset)^{[1]}\to \cat^{[1]}\times \sset^{[1]}$ sending a map of simplicial spaces $X\to Y$ to the tuple of $cRX\to cRY$ and $X_1\to Y_1\times_{Y_0^{\times2}} X_0^{\times 2}$. We define $\overline{\cW}$ to be the inverse image under $\Phi$ of the full subcategory $\cW_{\mathrm{Cat}}\times \cW_{\sset}$ of tuples of an equivalence of categories and a weak equivalence in $\sset$. Both of these are accessible subcategories since they form the weak equivalences of a combinatorial model structure. Since $\Phi$ preserves filtered colimits, it follows that $\overline{\cW}$ is an accessible subcategory as well by \cite[Corollary A.2.6.5]{Lurie}. Finally, \cref{lem:charhtpycat,charDK} imply that a map between Segal spaces $f\colon X\to Y$ is contained in $\overline{\cW}$ if and only if it is a Dwyer--Kan equivalence.
\end{proof}

\subsection{Path objects} \label{sec:pathobj}

We now show that Condition \ref{path} of \Cref{thm:main} holds. This follows from the following result, using \cref{Jfibration}.

\begin{prop} \label{prop:path}
    Let $X$ be a Segal space. The factorization of the diagonal morphism 
    \[ X \to X^{\NI} \to X\times X\]
    induced by $F[0]\amalg F[0]\hookrightarrow \NI\rightarrow F[0]$ is such that $X\to X^{\NI}$ is a Dwyer--Kan equivalence and $X^{\NI}\to X\times X$ is an isofibration. 
\end{prop}

\begin{proof}
The map $X^{\NI}\to X\times X$ is an isofibration by \cref{technicalisofib}, taking $Y=\Delta[0]$. The map $X \to X^{\NI}$ is a Dwyer--Kan equivalence by \cite[Lemma 13.9]{Rezk}.
\end{proof}

\subsection{Isofibrations that are Dwyer--Kan equivalences are trivial fibrations} \label{sec:fibwearetrivfib}

We now show that Condition \ref{fibwe} of \cref{thm:main} holds. In light of \cref{Jfibration}, this condition is precisely the implication ``(ii) implies (i)'' in the following statement.

\begin{prop} \label{proofof4}
   The following are equivalent for a map $f\colon X\to Y$ between Segal spaces:
    \begin{rome}
    \item the map $f$ is an $\cI$-fibration,
    \item the map $f$ is an isofibration and a Dwyer--Kan equivalence,
    \item the map $f$ is a Reedy fibration, it is homotopically fully faithful, and the induced map $X_{0, 0}\to Y_{0, 0}$ is surjective.
    \end{rome} 
\end{prop}

\begin{proof}
In all three cases, the map $f$ is in particular a Reedy fibration. 
The equivalence between (i) and (iii) then follows from \Cref{charDK} and \cref{Ifibration}.  If $f$ satisfies condition (iii), then $\ho X\to \ho Y$ is fully faithful and surjective on objects, so in particular an isofibration. Hence $f$ is an isofibration and a Dwyer--Kan equivalence. Conversely, suppose that $f$ is an isofibration and a Dwyer--Kan equivalence. To see that~$f$ satisfies (iii), the only nontrivial condition to check is that $X_{0, 0}\to Y_{0, 0}$ is surjective. This follows from the fact that the induced functor $\ho f$ is both an isofibration and an equivalence of categories, so it is surjective on objects.
\end{proof}

\subsection{Trivial fibrations are weak equivalences} \label{sec:trivfibarewe}

Finally, we show that Condition \ref{ax:trivfib} of \cref{thm:main} holds. For this, recall the coskeleton functor from \cref{not:cosk}.

\begin{prop} \label{pullbackalongIfib}
Let $p\colon K\to L$ be a fibration in $\sset$ between Kan complexes such that the induced map $p_0\colon K_0\to L_0$ is surjective. Then the functor 
\[
\cosk_0(p)^*\colon \Dsset_{\mathrm{Seg}\  /\cosk_0(L)}\to \Dsset_{\mathrm{Seg}\  /\cosk_0(K)}, \]
    obtained by taking pullback along the map $\cosk_0(p)\colon \cosk_0(K)\to \cosk_0(L)$, is left Quillen.
\end{prop}
\begin{proof}
 Recall that $\cosk_0(p)^*$ is a left adjoint. Moreover, it preserves monomorphisms, as it is also a right adjoint. In each degree, the map $\cosk_0(p)$ is given by the fibration $p^{\times m}\colon K^{\times m}\to L^{\times m}$ in $\sset$. Using that $\sset$ is right proper, we deduce that $\cosk_0(p)^*$ preserves levelwise weak equivalences of spaces. This shows that  
    \[ \cosk_0(p)^*\colon \Dsset_{\mathrm{Reedy}\ /\cosk_0(L)}\to \Dsset_{\mathrm{Seg}\ /\cosk_0(K)}\]
    is left Quillen. 
    To deduce that $\cosk_0(p)^*\colon \Dsset_{\mathrm{Seg}\ /\cosk_0(L)}\to \Dsset_{\mathrm{Seg}\ /\cosk_0(K)}$ is left Quillen, it suffices to show that it sends the inner horn inclusions $L^k[m]\hookrightarrow F[m]\to \cosk_0(L)$ for $0<k<m$ to weak equivalences in $\Dsset_{\mathrm{Seg}\ /\cosk_0(K)}$, by \cite[Theorem 3.3.19]{Hirschhorn} and \cite[Lemma 3.5]{JT}. Given $0<k<m$, consider the following pullback squares in $\Dsset$
\begin{tz}
        \node[](1) {$Q$}; 
        \node[right of=1,xshift=1cm](2) {$P$};
        \node[right of=2,xshift=1cm](2') {$\cosk_0(K)$};
        \node[below of=1](3) {$L^k[m]$}; 
        \node[below of=2](4) {$F[m]$}; 
        \node[below of=2'](4') {$\cosk_0(L)$}; 
        \pullback{2}; 
        \pullback{1};

        \draw[right hook->](1) to node[above, la]{$i$} (2); 
        \draw[right hook->](3) to (4); 
        \draw[->](2) to (2'); 
        \draw[->](4) to node[above, la]{$\alpha$} (4'); 
        \draw[->>](1) to (3); 
        \draw[->>](2) to (4);  
        \draw[->>](2') to node[right,la]{$\cosk_0(p)$} (4');
    \end{tz}
    Let us first show that $i$ is a trivial cofibration in $ \Dsset_{\mathrm{Seg}}$ when $K$ and $L$ are sets, so that all objects in the above diagram are contained in the subcategory $\set^{\Delta^{\op}}\subseteq \Dsset$. Let us write $K(i)=p^{-1}(\alpha(i))$ for the fibre of $p$ over the image of the $i$-th vertex of $F[m]$ in $\cosk_0(L)_0 = L$. Then $P=N\cC$ is the nerve of the category with:
    \begin{itemize}
        \item objects given by tuples $(i, x)$ with $0\leq i\leq m$ and $x\in K(i)$.
        \item exactly one morphism $(i, x)\to (j, y)$ whenever $i\leq j$, and no morphism otherwise.
    \end{itemize}
    Let $\sigma\colon (i_0, x_0)\to \dots \to (i_n, x_n)$ be an $n$-simplex in $P$. Then $\sigma$ is \emph{not} contained in $Q$ if and only if $\{0, \dots, \hat{k}, \dots, m\}\subseteq \{i_0, \dots, i_n\}$. Given such a non-degenerate chain that is not contained in $Q$, let us define its \emph{pivot} to be the first element $(i_a, x_a)$ in the chain with $i_a=k$; note that $\sigma$ may not have a pivot.
    
    Let us now fix an element $z\in K(k)$ in the (non-empty) fibre $K(k)$ and say that a non-degenerate chain $\sigma$ not contained in $Q$ is:
    \begin{itemize}
        \item \emph{bounding} if it has a pivot of the form $(k, x_a=z)$, and
        \item \emph{bounded} if not. In particular, a chain without any term of the form $(k, x)$ is bounded.
    \end{itemize}
    Every bounded chain $\sigma$ is an inner face of exactly one bounding chain $\tau$, obtained from $\sigma$ by adding a pivot $(k, z)$ as follows, where $i_a=k$ or $i_a=k+1$.
    \begin{tz}
        \node[](1) {$\sigma\colon (0, x_0)$}; 
        \node[right of=1,xshift=.2cm](2) {$\dots$};
        \node[right of=2, xshift=.5cm](3) {$(k-1, x_{a-1})$};
        \node[right of=3, xshift=.7cm](4) {};
        \node[right of=4, xshift=0.2cm](5) {$(i_a, x_{a})$};
        \node[right of=5](6) {$\dots$};
        \node[right of=6](7) {$(m, x_n)$};
        
        \node[below of=1, yshift=1cm](1') {$\tau\colon (0, x_0)$}; 
        \node[right of=1', xshift=.2cm](2') {$\dots$};
        \node[right of=2', xshift=.5cm](3') {$(k-1, x_{a-1})$};
        \node[right of=3', xshift=.7cm](4') {$(k, z)$};
        \node[right of=4', xshift=0.2cm](5') {$(i_a, x_{a})$};
        \node[right of=5'](6') {$\dots$};
        \node[right of=6'](7') {$(m, x_n)$};
               
        \draw[->](1) to (2); 
        \draw[->](2) to (3); 
        \draw[->](3) to (5); 
        \draw[->](5) to (6);
        \draw[->](6) to (7); 
        \draw[->](1') to (2'); 
        \draw[->](2') to (3'); 
        \draw[->](3') to (4'); 
        \draw[->](4') to (5'); 
        \draw[->](5') to (6');
        \draw[->](6') to (7'); 
    \end{tz}
    Conversely, if $\tau$ is a (by definition non-degenerate) bounding chain with pivot $(i_a, x_a)=(k, z)$, then $d_j(\tau)$ is a lower-dimensional bounding chain or contained in $P$ when $j\neq a$, while the inner face $d_a(\tau)$ is a bounded chain of lower dimension.
    It follows from this that $Q\hookrightarrow P$ is an iterated pushout of inner horn inclusions: proceeding by induction on the dimension, we can take pushouts along $L^a[n]\to F[n]$ to add a bounding chain $\tau$ of length $n$, together with its (unique) bounded inner face $d_a(\tau)$. Since every iterated pushout of inner horn inclusions is a trivial cofibration in $ \Dsset_{\mathrm{Seg}}$ by \cite[Lemma 3.5]{JT}, it follows that $Q\hookrightarrow P$ is a trivial cofibration.
    
    Finally, let us treat the case where $p\colon K\to L$ is a map of simplicial sets. Since $p$ is a Kan fibration and surjective on vertices, each map $p_n\colon K_n\to L_n$ is surjective. For $n\geq 0$, we regard $P_{-,n}$ and $Q_{-,n}$ as objects of $\sset^{\Dop}$ through the canonical inclusion $\set^{\Dop}\hookrightarrow \sset^{\Dop}$. Applying the previous argument, the induced map
    \[ Q_{-,n}=L^k[m]\times_{\cosk_0(L_n)} \cosk_0(K_n)\to F[m]\times_{\cosk_0(L_n)} \cosk_0(K_n)=P_{-,n} \]
    is a weak equivalence in $\Dsset_{\mathrm{Seg}}$. Moreover, these maps assemble into a natural transformation of functors $\Dop\to \sset^{\Dop}$ and the induced map $\mathrm{hocolim}_{[n]\in \Dop} Q_{-,n} \to \mathrm{hocolim}_{[n]\in \Dop} P_{-,n}$ between homotopy colimits is also a weak equivalence in $\Dsset_{\mathrm{Seg}}$. Since $\Dsset_{\mathrm{Seg}}$ is a simplicial model structure for the levelwise tensoring over $\sset$, these homotopy colimits can be computed by the diagonal by \cite[Theorem 15.11.6 and Theorem 18.7.6]{Hirschhorn}. Using that $P_{-,n}$ and $Q_{-,n}$ are constant in the space direction, we conclude that $i\colon Q\hookrightarrow P$ is a weak equivalence, as desired.
\end{proof}

\begin{lemma} \label{coskisIinj}
    The following assertions hold.
    \begin{rome}
        \item Let $p\colon K\to L$ be a fibration in $\sset$ between Kan complexes such that the induced map $p_{0}\colon K_{0}\to L_{0}$ is surjective. Then the induced map 
    \[ \cosk_0(p)\colon \cosk_0(K)\to \cosk_0(L) \]
    is an $\cI$-fibration between Segal spaces. 
        \item Let $f\colon X\to Y$ be an $\cI$-fibration. Then the map $X\to Y\times_{\cosk_0(Y_0)} \cosk_0(X_0)$ is a Reedy trivial fibration.
    \end{rome}
    
\end{lemma}

\begin{proof}
    This is straightforward from unpacking the $m$-th relative matching maps and Segal maps, and using \cref{Ifibration}. The $m$-th relative matching map of $X\to Y\times_{\cosk_0(Y_0)} \cosk_0(X_0)$ is isomorphic to the $m$-th matching map of $X\to Y$ if $m\geq 1$, and an isomorphism when $m=0$.
\end{proof}

We are now ready to show that Condition \ref{ax:trivfib} of \cref{thm:main} holds. 

\begin{prop}
    Let $f\colon X\to Y$ be an $\cI$-fibration. Then there is a $\cJ$-fibrant replacement of $f$ that is a Dwyer--Kan equivalence.
\end{prop}

\begin{proof}
    By \cref{lem:cJ0}, we can choose the $\cJ$-fibrant replacements to be $\cJ_0$-fibrant replacements. Let $j_Y\colon Y\to \widetilde{Y}$ be a $\cJ_0$-fibrant replacement and factor the composite $j_Y\circ f$ as a $\cJ_0$-cofibration $j_X$ followed by a $\cJ_0$-fibration~$\widetilde{f}$, as below left. Then, consider the below right commutative diagram in $\Dsset$.
    \begin{tz}
        \node[](1) {$X$}; 
        \node[right of=1](2) {$Y$}; 
        \node[below of=1](3) {$\widetilde{X}$}; 
        \node[below of=2](4) {$\widetilde{Y}$}; 

        \draw[->](1) to node[above,la]{$f$} (2); 
        \draw[->](3) to node[below,la]{$\widetilde{f}$} (4); 
        \draw[->](1) to node[left,la]{$\cJ_0\text{-}\mathrm{cof}\ni j_X$} (3); 
        \draw[->](2) to node[right,la]{$j_Y\in \cJ_0\text{-}\mathrm{cof}$} (4); 
        
        \node[right of=2,xshift=2cm](1) {$X$}; 
        \node[below of=1](2) {$Y$}; 
        \node[right of=1,xshift=1cm](1') {$\cosk_0(X_0)$}; 
        \node[below of=1'](2') {$\cosk_0(Y_0)$}; 
        \node[right of=1',xshift=2cm](1'') {$\cosk_0(\widetilde{X}_0)$};
        \node[below of=1''](2'') {$\cosk_0(\widetilde{Y}_0)$};

        \draw[->>](1) to node[left,la]{$f$} (2); 
        \draw[->>](1') to node[left,la]{$\cosk_0(f_0)$} (2'); 
        \draw[->>](1'') to node[right,la]{$\cosk_0(\widetilde{f}_0)$} (2'');
        \draw[->](1) to (1'); 
        \draw[->](1') to node[above,la]{$\cosk_0((j_X)_0)$} (1''); 
        \draw[->](2) to (2'); 
        \draw[->](2') to node[below,la]{$\cosk_0((j_Y)_0)$} (2''); 
    \end{tz}
    Since $j_X$ and $j_Y$ are $\cJ_0$-cofibrations, the induced maps $(j_X)_0\colon X_0\to \widetilde{X}_0$ and $(j_Y)_0\colon Y_0\to \widetilde{Y}_0$ are weak equivalences in $\sset$ by \cref{J0coflevel01}. We therefore obtain levelwise weak equivalences $\cosk_0((j_X)_0)$ and $\cosk_0((j_Y)_0)$. In the right diagram above, the right-hand square is therefore levelwise a homotopy pullback in $\sset$. The left-hand square is also levelwise a homotopy pullback in $\sset$: the right vertical map is levelwise a fibration in $\sset$ (which is right proper) and $X\to Y\times_{\cosk_0(Y_0)} \cosk_0(X_0)$ is a levelwise weak equivalence in $\sset$ by Lemma \ref{coskisIinj} (ii). It follows that the composite of the two squares is levelwise a homotopy pullback in $\sset$, i.e., the map 
    \[ X\to Y\times_{\cosk_0(\widetilde{Y}_0)} \cosk_0(\widetilde{X}_0) = P\]
    is a levelwise weak equivalence. Now consider the induced diagram in $\Dsset$
    \begin{tz}
    \node[](0) {$X$}; 
    \node[right of =0,xshift=1cm](0') {$\widetilde{X}$};
        \node[below of=0,yshift=.5cm,xshift=1cm](1) {$P$}; 
        \node[below of=1](2) {$Y$}; 
        \node[below of=0',yshift=.5cm,xshift=1cm](1') {$Q$}; 
        \node[below of=1'](2') {$\widetilde{Y}$}; 
        \node[right of=1',xshift=1cm](1'') {$\cosk_0(\widetilde{X}_0)$};
        \node[below of=1''](2'') {$\cosk_0(\widetilde{Y}_0)$};
        \pullback{1}; 
        \pullback{1'};

        \draw[->](0) to (1); 
        \draw[->](0') to (1'); 
        \draw[->>](1) to (2); 
        \draw[->>](1') to (2'); 
        \draw[->>](1'') to node[right,la]{$\mathrm{cosk}(\widetilde{f}_0)$} (2'');
        \draw[->](0) to node[above,la]{$j_X$} (0');  
        \draw[->](1') to (1''); 
        \draw[->](2) to node[below,la]{$j_Y$} (2'); 
        \draw[->](2') to (2''); 
        \draw[->,bend right=20](0) to node[left,la,yshift=-8pt]{$f$} (2);
        \draw[->,bend right=20](0') to node[left,la,yshift=-8pt,xshift=-1pt]{$\widetilde{f}$} (2');
        
        \draw[-,line width=2mm,white](1) to (1');
        \draw[->](1) to (1');
    \end{tz}
    where $P$ and $Q$ are defined by taking pullbacks. Because $\widetilde{f}\colon \widetilde{X}\to \widetilde{Y}$ is a $\cJ_0$-fibration between $\cJ_0$-fibrant objects, \cref{rem:J0fib,J0coflevel01} show that the map $\widetilde{f}_0\colon \widetilde{X}_0\to \widetilde{Y}_0$ is a (Kan) fibrant replacement of the fibration $f_0$ in $\sset$. Since $f_{0,0}$ is surjective, so are the isomorphic maps $\pi_0(f_0)$ and $\pi_0(\widetilde{f}_0)$. Since $\widetilde{f}_0$ a fibration in $\sset$, the map $\widetilde{f}_{0,0}$ is then surjective as well. \Cref{coskisIinj} (i) therefore implies that $\cosk_0(\widetilde{f}_0)$ is an $\cI$-fibration, so that its pullback $Q\to \widetilde{Y}$ is an $\cI$-fibration between Segal spaces. Consequently, it is a Dwyer--Kan equivalence by \cref{proofof4}.    
    
    In addition, \cref{rem:J0fib,pullbackalongIfib} imply that $P\to Q$ is a weak equivalence in $\Dsset_\mathrm{Seg}$, as it is the pullback of the $\cJ_0$-cofibration $j_Y\colon Y\to \widetilde{Y}$ along $\cosk_0(\widetilde{f}_0)$. The map $X\to P$ is a levelwise weak equivalence by the above argument and the map $j_X\colon X\to \widetilde{X}$ is a $\cJ_0$-cofibration, and so a weak equivalence in $\Dsset_\mathrm{Seg}$ by \cref{rem:J0fib}. It follows from $2$-out-of-$3$ that $\widetilde{X}\to Q$ is a weak equivalence between Segal spaces in $\Dsset_\mathrm{Seg}$. By \cref{levelwiseDK}, the map $\widetilde{X}\to Q$ is then a Dwyer--Kan equivalence, so that the composite $\widetilde{f}\colon \widetilde{X}\to \widetilde{Y}$ is a Dwyer--Kan equivalence. 
\end{proof}

\section{Quillen equivalences with models of \texorpdfstring{$(\infty,1)$}{(infinity,1)}-categories} \label{sec:QE}

In this section, we show that the categorical model structure $\Dsset_{\mathrm{Cat}}$ constructed in \cref{thm:MS} is a model of $(\infty,1)$-categories. Specifically, in \cref{sec:QECSS}, we prove that it is Quillen equivalent to the complete Segal space model structure on 
$\Dsset$ via the identity adjunction. Additionally, in \cref{sec:QEPcat}, we compare it with the model of $(\infty,1)$-categories given by the \emph{Segal categories}. Denoting by $\cP\cat(\sset)$ the full subcategory of $\Dsset$ consisting of the simplicial spaces $X$ such that $X_0$ is a set, we show that the inclusion $\cP\cat(\sset) \hookrightarrow \Dsset_\mathrm{Cat}$ is both a left and a right Quillen equivalence when $\cP\cat(\sset)$ is endowed with the Segal category model structure. As a consequence, we get that the Segal category model structure is both left- and right-induced from the categorical model structure $\Dsset_{\mathrm{Cat}}$.

\subsection{Quillen equivalence with complete Segal spaces} \label{sec:QECSS}

We begin by showing that the categorical model structure that we constructed in \cref{thm:MS} is Quillen equivalent to Rezk's model for $(\infty,1)$-categories, given by the \emph{complete} Segal spaces. To this end, we first review their definition, as well as their associated model structure.

\begin{defn}
    A simplicial space $X\colon \Dop \to \sset$ is a \emph{complete Segal space} if it is a Segal space, and the completeness map $\Map(\NI,X)\to \Map(F[0],X)\cong X_0$ is a weak equivalence in $\sset$. 
\end{defn}

The following is a combination of \cite[Theorems 7.2 and 7.7]{Rezk}.

\begin{theorem} \label{thm:MSCSS}
    There is a cofibrantly generated model structure on the category $\Dsset$, which we denote by $\Dsset_{\mathrm{CSS}}$, in which 
    \begin{rome}
        \item the cofibrations are the monomorphisms, 
        \item the fibrant objects are the complete Segal spaces,
        \item the weak equivalences between Segal spaces are the Dwyer--Kan equivalences.
    \end{rome}
\end{theorem}

We now prove the desired result, using the following lemma.

\begin{lemma} \label{lem:weinCSStoSegal}
    If $f\colon X\to Y$ is a map in $\Dsset$ such that $f$ is a weak equivalence in $\Dsset_\mathrm{CSS}$, then $f$ is a weak equivalence in $\Dsset_\mathrm{Cat}$.
\end{lemma}

\begin{proof}
    Let $j_Y\colon Y\to \widetilde{Y}$ be a $\cJ$-fibrant replacement and factor the composite $j_Yf$ into a $\cJ$-cofibration $j_X\colon X\to \widetilde{X}$ followed by a $\cJ$-fibration $\tilde{f}\colon \widetilde{X}\to \widetilde{Y}$. Since $\cJ$-cofibrations are trivial cofibrations in $\Dsset_\mathrm{CSS}$, it follows by $2$-out-of-$3$ that $\tilde{f}$ is a weak equivalence in $\Dsset_\mathrm{CSS}$ between Segal spaces. By \cref{thm:MSCSS} (iii), the map $\tilde{f}$ is a Dwyer--Kan equivalence between Segal spaces, and hence a weak equivalence in  $\Dsset_\mathrm{Cat}$. As $\cJ$-cofibrations are in particular weak equivalences in  $\Dsset_\mathrm{Cat}$, the map $f$ is a weak equivalence in  $\Dsset_\mathrm{Cat}$ by the $2$-out-of-$3$ property.
\end{proof}

\begin{theorem} \label{thm:idQE}
    The identity adjunction
    \[ \id\colon \Dsset_{\mathrm{Cat}} \rightleftarrows \Dsset_{\mathrm{CSS}}\cocolon \id \]
    is a Quillen equivalence. 
\end{theorem}

\begin{proof}
    By \cref{Icof}, the functor $\id\colon \Dsset_{\mathrm{Cat}} \to \Dsset_{\mathrm{CSS}}$ preserves cofibrations. By \cref{Jfibration}, the functor $\id\colon \Dsset_{\mathrm{CSS}}\to \Dsset_{\mathrm{Cat}}$ preserves fibrations between fibrant objects. Hence the identity adjunction is a Quillen pair by \cite[Proposition E.2.14]{JoyalVolumeII}.  

    Next, we show that the derived unit is a weak equivalence. Let $X$ be a cofibrant object in $\Dsset_{\mathrm{Cat}}$ and consider a fibrant replacement $X\to \widehat{X}$ in $\Dsset_{\mathrm{CSS}}$. Then by \cref{lem:weinCSStoSegal} the component of the derived unit $X\to \widehat{X}$ is a weak equivalence in $\Dsset_{\mathrm{Cat}}$.

    Finally, we show that the derived counit is a weak equivalence. Let $X$ be a fibrant object in $\Dsset_{\mathrm{CSS}}$ and consider a cofibrant replacement $\overline{X}\to X$ in $\Dsset_{\mathrm{Cat}}$ given by an $\cI$-fibration. We want to prove that the component of the derived counit $\overline{X}\to X$ is a weak equivalence in $\Dsset_{\mathrm{CSS}}$. Since the complete Segal space $X$ is in particular a Segal space, then so is $\overline{X}$. Hence, by \cref{proofof4}, the map $\overline{X}\to X$ is a Dwyer--Kan equivalence between Segal spaces, and so it is a weak equivalence in $\Dsset_{\mathrm{CSS}}$ by \cref{thm:MSCSS} (iii). 
\end{proof}

\subsection{Quillen equivalences with Segal categories} \label{sec:QEPcat}

We now compare the categorical model structure with another model of $(\infty,1)$-categories, given by the \emph{Segal categories}. Let us first review the main features of the model structure for Segal categories, whose existence is established in \cite[Theorem 5.1]{Bergner}. The characterization of the fibrant objects follows from \cite[Corollary 5.13]{Bergner} and \cite[Theorem 3.2]{Bergnerfibrant}.

\begin{theorem}
    The category $\cP\cat(\sset)$ admits a cofibrantly generated model structure, in which
    \begin{rome}
        \item the cofibrations are the monomorphisms, 
        \item the fibrant objects are the Segal categories, i.e., Segal spaces with a discrete space of objects,
        \item the weak equivalences between Segal spaces are the Dwyer--Kan equivalences. 
    \end{rome}
\end{theorem}

\begin{rem} \label{fibreplinpcat}
    By \cite[\textsection 5]{Bergner}, a fibrant replacement of an object $X\in \cP\cat(\sset)$ can be computed in $\Dsset$ using the weak factorization system generated by the maps of the form
    \[ \partial F[m]\times \Delta[n]\amalg_{\partial F[m]\times \Lambda^k[n]} F[m]\times \Lambda^k[n]\hookrightarrow F[m]\times \Delta[n], \]
    for $m\geq 1$ and $0\leq k\leq n$, and the maps of the form 
    \[ \Sp[m]\times \Delta[n]\amalg_{\Sp[m]\times \partial\Delta[n]} F[m]\times \partial\Delta [n]\hookrightarrow F[m]\times \Delta[n], \]
    for $m\geq 2$ and $n\geq 0$. Then a map $f\colon X\to Y$ in $\cP\cat(\sset)$ is defined to be a weak equivalence if its fibrant replacement is a Dwyer--Kan equivalence.
\end{rem}

The Segal category model structure is also a model of $(\infty,1)$-categories, as it is Quillen equivalent to the complete Segal space model structure by \cite[Theorem 6.3]{Bergner}. This Quillen equivalence is induced by the inclusion $I\colon \cP\cat(\sset)\hookrightarrow \Dsset$, which we recall admits both a left adjoint $L\colon \Dsset\to \cP\cat(\sset)$ and a right adjoint $R\colon \Dsset\to \cP\cat(\sset)$, see \cref{rem:reduction}. 

\begin{theorem} \label{IRQE}
    The adjunction $I\colon \cP\cat(\sset)\rightleftarrows \Dsset_{\mathrm{CSS}}\cocolon R$
    is a Quillen equivalence.
\end{theorem}

\begin{rem}
    The adjunction $L\dashv I$ is however \emph{not} a Quillen equivalence between $\Dsset_\mathrm{CSS}$ and $\cP\cat(\sset)$, as $L$ does not preserves monomorphisms, and so is not left Quillen.
\end{rem}

In the categorical model structure that we constructed in \cref{thm:MS}, we removed all the problematic monomorphisms whose image under $L$ is not a monomorphism. As a consequence, the adjunction $L\dashv I$ does become a Quillen equivalence between $\Dsset_\mathrm{Cat}$ and $\cP\cat(\sset)$. We prove the following.

\begin{theorem} \label{LIQE}
    The adjunctions 
    \[ L\colon\Dsset_{\mathrm{Cat}} \rightleftarrows \cP\cat(\sset)\cocolon I \quad \text{and} \quad I\colon\cP\cat(\sset) \rightleftarrows \Dsset_{\mathrm{Cat}}\cocolon R \]
    are Quillen equivalences.
\end{theorem}

To prove this result, we first show that the adjunctions $L\dashv I$ and $I\dashv R$ are Quillen pairs.

\begin{prop}
    The adjunction $L\colon \Dsset_{\mathrm{Cat}}\rightleftarrows \cP\cat(\sset)\cocolon I$ is a Quillen pair. 
\end{prop}

\begin{proof}
    Recall from \cite[\textsection 4]{Bergner} that a set of generating cofibrations for the model structure on $\cP\cat(\sset)$ is given by the set containing the maps
    \[ L\left(\partial F[m]\times \Delta[n]\amalg_{\partial F[m]\times \partial\Delta[n]} F[m]\times \partial\Delta [n]\hookrightarrow F[m]\times \Delta[n]\right), \]
    for $m=n=0$ and for $m\geq 1$ and $n\geq 0$. Further using that $L(\Lambda^k[n]\hookrightarrow \Delta[n])$ is the identity at $F[0]$, we see that the functor $L\colon \Dsset_\mathrm{Cat}\to \cP\cat(\sset)$ preserves cofibrations. 
    
    By \cite[Proposition E.2.14]{JoyalVolumeII}, it remains to show that $L\colon \Dsset_\mathrm{Cat}\to \cP\cat(\sset)$ sends maps in $\cJ$ to weak equivalences in $\cP\cat(\sset)$.
    First, consider a map $X\hookrightarrow Y$ in $\cJ_0$, i.e., a map of the form (i) or (ii) in \cref{notn:J}. Using \cref{fibreplinpcat}, we see that the fibrant replacement of the induced map $LX\hookrightarrow LY$ can be computed to be the identity, and so the map $LX\hookrightarrow LY$ is a weak equivalence in $\cP\cat(\sset)$. Finally, if $X\hookrightarrow Y$ is the map $F[0]\hookrightarrow \NI$, then it is a Dwyer--Kan equivalence between Segal spaces and so it is also a weak equivalence in~$\cP\cat(\sset)$. 
\end{proof}

\begin{prop}
    The adjunction $I\colon\cP\cat(\sset) \rightleftarrows \Dsset_{\mathrm{Cat}}\cocolon R$ is a Quillen pair. 
\end{prop}

\begin{proof}
    Suppose that $X\hookrightarrow Y$ is a monomorphism in $\cP\cat(\sset)$. Since $I$ is a right adjoint and $X_0$ and $Y_0$ are sets, the monomorphism $IX\hookrightarrow IY$ satisfies condition (2) of \cref{Icof}, hence it is a cofibration. This shows that $I\colon \cP\cat(\sset)\to \Dsset_{\mathrm{Cat}}$ preserves cofibrations. 

    We now prove that $I\colon \cP\cat(\sset)\to \Dsset_{\mathrm{Cat}}$ preserves weak equivalences. By \cref{fibreplinpcat}, we see that $I$ preserves fibrant replacements. Hence, it suffices to show that $I$ preserves weak equivalences between fibrant objects, but this is clear as in both model structures these are given by the Dwyer--Kan equivalences between Segal spaces. 
\end{proof}

To prove that the above are further Quillen equivalences, we show that the following diagram of right Quillen functors induces a commutative diagram at the level of homotopy categories. We then deduce from \cref{thm:idQE,IRQE} and the $2$-out-of-$3$ property for Quillen equivalences that $I$ is a Quillen equivalence, thereby concluding the proof of \cref{LIQE}. 

\begin{prop}
    The counit of the adjunction $I\dashv R$ induces a diagram of right Quillen functors which commutes at fibrant objects up to a weak equivalence in $\Dsset_\mathrm{Cat}$.
    \begin{tz}
    \node[](1) {$\Dsset_\mathrm{CSS}$}; 
    \node[above of=1,xshift=2.5cm,yshift=-.3cm](2) {$\cP\cat(\sset)$}; 
    \node[below of=2,xshift=2.5cm,yshift=.3cm](3) {$\Dsset_\mathrm{Cat}$}; 

    \draw[right hook->](2) to node[above,la,xshift=3pt]{$I$} (3);
    \draw[->](1) to node(a)[below,la]{$\id$} (3);
    \draw[->](1) to node[above,la,xshift=-3pt]{$R$} (2); 

    \node at ($(2)!0.5!(a)$) {\rotatebox{270}{$\Rightarrow$}};
    \node[la] at ($(2)!0.5!(a)-(7pt,0)$) {$\varepsilon$};
\end{tz}
\end{prop}

\begin{proof}
    Let $X$ be a fibrant object in $\Dsset_\mathrm{CSS}$, i.e., a complete Segal space. Since all objects in $\cP\cat(\sset)$ are cofibrant, the component of the counit $\varepsilon_X\colon IRX\to X$ coincides with the component of the derived counit at $X$. As $IRX$ and $X$ are both Segal spaces and $\varepsilon_X$ induces a bijection $(IRX)_{0, 0}\cong X_{0, 0}$, it follows directly from \cref{rem:reduction,charDK} that $\varepsilon_X\colon IRX\to X$ is a Dwyer--Kan equivalence, and so a weak equivalence in~$\Dsset_\mathrm{Cat}$. 
\end{proof}

As a consequence, we get the following result, from which we directly extract a characterization of the fibrations between fibrant objects in the Segal category model structure. 

\begin{cor} \label{cor:rightind}
    The model structure $\cP\cat(\sset)$ is left- and right-induced from the model structure $\Dsset_\mathrm{Cat}$ along the inclusion $I\colon \cP\cat(\sset)\to \Dsset$.
\end{cor}

\begin{cor}
    The fibrations between Segal spaces in $\cP\cat(\sset)$ are the isofibrations. 
\end{cor}

We recover \cref{cor:rightind} as an application of the following abstract argument, by taking $\cM = \cP\cat(\sset)$, $\cN = \Dsset_\mathrm{Cat}$, and $F = I$.

\begin{lemma}
    Let $\cM$ and $\cN$ be model categories and $F\colon \cM\to \cN$ be a fully faithful functor that is both a left and a right Quillen equivalence. Then the model structure on~$\cM$ is both left- and right-induced along $F$ from that on $\cN$. 
\end{lemma}

\begin{proof}
    Since $F\colon \cM\to \cN$ is both left and right Quillen, it preserves (trivial) fibrations and (trivial) cofibrations. We further prove that it reflects (trivial) fibrations and (trivial) cofibrations. Then, since a map is a weak equivalence if and only if it factors as a trivial cofibration followed by a trivial fibration, then $F$ also preserves and reflects all weak equivalences.
    
    We only show that $F$ reflects fibrations; the case of trivial fibrations is similar, and the case of (trivial) cofibrations is dual. Let $X\to Y$ be a map in $\cM$ such that the induced map $FX\twoheadrightarrow FY$ is a fibration in $\cN$. To see that $X\to Y$ is a fibration in $\cM$, we need to show that it has the right lifting property against every trivial cofibration $A\xhookrightarrow{\simeq} B$ in~$\cM$. Since $F$ is fully faithful, this is equivalent to showing that $FX\twoheadrightarrow FY$ has the right lifting property against $FA\to FB$, for every trivial cofibration $A\xhookrightarrow{\simeq} B$ in $\cM$. This follows from the fact that $F$ preserves trivial cofibrations.
\end{proof}

\section{Properties of the model structure}

In this section, we show that the categorical model structure $\Dsset_{\mathrm{Cat}}$ constructed in \cref{thm:MS} has desirable properties. Specifically, in \cref{sec:cartclosed}, we prove that it is cartesian closed, and in \cref{sec:leftproper}, we establish that it is left proper.

\subsection{Cartesian closedness} \label{sec:cartclosed}

We first prove the following.

 \begin{theorem} \label{thm:cartclosed}
     The model category $\Dsset_\mathrm{Cat}$ from \cref{thm:MS} is cartesian closed.
 \end{theorem}

 Let us start by showing that the pushout-product of two $\cI$-cofibrations is an $\cI$-cofibration.

 \begin{prop}
     Let $f\colon X\to Y$ and $f'\colon X'\to Y'$ be $\cI$-cofibrations. The pushout-product map
     \[ f\,\square\,f'\colon X\times Y'\amalg_{X\times X'} Y\times X'\to Y\times Y' \]
     is an $\cI$-cofibration. 
 \end{prop}

 \begin{proof}
     By \cref{Icof}, the $\cI$-cofibrations $f\colon X\to Y$ and $f'\colon X'\to Y'$ are monomorphisms and there are sets $R$ and $R'$ and weak equivalences $X_0\amalg R\xrightarrow{\simeq} Y_0$ and $X'_0\amalg R'\xrightarrow{\simeq} Y'_0$ in $\sset$ whose restrictions to $X_0$ and $X'_0$ are $f_0$ and $f'_0$, respectively. We show that the pushout-product map $f\,\square\,f'$ is an $\cI$-cofibration by verifying the conditions from \cref{Icof}. 

     Since the maps $f$ and $f'$ are monomorphisms, so is their pushout-product $f\,\square\,f'$. Consider the following maps in $\sset$
     \begin{tz}
         \node[](1) {$\left( X_0\times (X'_0\amalg R')\amalg_{X_0\times X'_0} (X_0\amalg R)\times X'_0\right)\amalg R\times R'$};
         \node[below of=1,yshift=-.2cm](2) {$(X_0\amalg R)\times (X'_0\amalg R')$}; 
         \node(2') at ($(1)!0.55!(2)$) {$\left( X_0\times X'_0\amalg X_0\times R'\amalg R\times X'_0\right) \amalg R\times R'$};
         \node[right of=2,xshift=7cm](3) {$Y_0\times Y'_0$}; 
         \node[above of=3,yshift=.2cm](4) {$\left( X_0\times Y'_0\amalg_{X_0\times X'_0} Y_0\times X'_0\right)\amalg R\times R'$};
         
         \node at ($(1)!0.5!(2')$) {\rotatebox{270}{$\cong$}};
         \node at ($(2)!0.5!(2')$) {\rotatebox{270}{$\cong$}};
         \draw[->](1) to node[above,la]{$\simeq$} (4);
         \draw[->](2) to node[above,la]{$\simeq$} (3);
         \draw[->](4) to (3);
     \end{tz}
     The top map is a weak equivalence in $\sset$ as a map between homotopy pushouts induced from a diagram of weak equivalences, and the bottom map is a weak equivalence in $\sset$ since this is cartesian closed. Hence, by $2$-out-of-$3$, the right-hand map is a weak equivalence in $\sset$ whose restriction to $X_0\times Y'_0\amalg_{X_0\times X'_0} Y_0\times X'_0$ is $(f\,\square\,f')_0$.
     This shows that $f\,\square\,f'$ is an $\cI$-cofibration. 
 \end{proof}

 We now further show that, if one of the two $\cI$-cofibrations is a $\cJ$-cofibration, then their pushout-product is in fact a weak equivalence in $\Dsset_\mathrm{Cat}$. By \cite[Proposition E.2.14]{JoyalVolumeII}, this is sufficient to conclude the proof of \cref{thm:cartclosed}. 

  \begin{prop}
     Let $f\colon X\to Y$ be an $\cI$-cofibration and $f'\colon X'\to Y'$ be a $\cJ$-cofibration. The pushout-product map
     \[ f\,\square\,f'\colon X\times Y'\amalg_{X\times X'} Y\times X'\to Y\times Y' \]
     is a weak equivalence in $\Dsset_\mathrm{Cat}$.
 \end{prop}

 \begin{proof}
    We first deal with the case where $f'$ is a $\cJ_0$-cofibration, and hence a trivial cofibration in $\Dsset_\mathrm{Seg}$ by \cref{rem:J0fib}. By \cref{Icof}, the $\cI$-cofibration $f\colon X\to Y$ is a monomorphism, and so a cofibration in $\Dsset_\mathrm{Seg}$. Since $\Dsset_\mathrm{Seg}$ is cartesian closed, we get that the pushout-product map $f\,\square\, f'$ is weak equivalence in $\Dsset_\mathrm{Seg}$. By $2$-out-of-$3$, a $\cJ_0$-fibrant replacement of $f\,\square\,f'$ is also a weak equivalence in $\Dsset_\mathrm{Seg}$ between Segal spaces, and so a Dwyer--Kan equivalence by \cref{levelwiseDK}. This shows that $f\,\square\,f'$ is a weak equivalence. 

   To treat the general case, it is now enough to treat the case $f'$ being $F[0]\hookrightarrow \NI$, since $\cJ$-cofibrations are generated by maps in $\cJ_0$ and $F[0]\hookrightarrow \NI$. Showing that the $\cI$-cofibration $X\times \NI\amalg_{X} Y\to Y\times \NI$ is a weak equivalence is equivalent to showing that, for every $\cJ$-fibration $W\to Z$ between $\cJ$-fibrant objects, the induced map 
   \[ W^{\NI}\to Z^{\NI}\times_Z W \]
   is an $\cI$-fibration. We prove this by showing that it is both an isofibration and a Dwyer--Kan equivalence. Since $W\to Z$ is an isofibration between Segal spaces, then \cref{technicalisofib} shows that the desired map is an isofibration. Moreover, by \cref{technicalisofib,prop:path}, the maps $W^{\NI}\to W$ and $Z^{\NI}\to Z$ are both isofibrations and Dwyer--Kan equivalences; hence they are $\cI$-fibrations. Then, in the composite
   \[ W^{\NI}\to Z^{\NI}\times_Z W \to W, \]
   the second map is an $\cI$-fibration as the pullback of $Z^{\NI}\to Z$ and the composite is a Dwyer--Kan equivalence. Hence, the first map is a Dwyer--Kan equivalence as well, by $2$-out-of-$3$. 
 \end{proof}

 \subsection{Left properness} \label{sec:leftproper}

Next, we show that the categorical model structure is left proper.

\begin{theorem} \label{thm:leftproper}
    The model category $\Dsset_\mathrm{Cat}$ from \cref{thm:MS} is left proper.
\end{theorem}

\begin{proof}
    Let $i\colon A\to B$ be an $\cI$-cofibration and $f\colon A\to C$ be a weak equivalence $\Dsset_\mathrm{Cat}$. Consider the pushout $g\colon B\to D$ of 
    $f$ along $i$; we want to prove that $g$ is also a weak equivalence in $\Dsset_\mathrm{Cat}$. First note that, without loss of generality, we can assume that $C$ is a Segal space, as $\cJ$-cofibrations are closed under pushouts and weak equivalences in $\Dsset_\mathrm{Cat}$ satisfy $2$-out-of-$3$. 
    
    Next, we factor the weak equivalence $f$ into a $\cJ$-cofibration $j$ followed by a $\cJ$-fibration $p$, so that the below left pushout square factors as the composite of pushout squares in $\Dsset$. 
    \begin{tz}
        \node[](1) {$A$}; 
        \node[right of=1](2) {$C$}; 
        \node[below of=1](3) {$B$}; 
        \node[below of=2](4) {$D$}; 

        \draw[->](1) to node[above,la]{$f$} (2); 
        \draw[->](3) to node[below,la]{$g$} (4); 
        \draw[->](1) to node[left,la]{$i$} (3); 
        \draw[->](2) to node[right,la]{} (4);
        \pushout{4};
        
        \node[right of=2,xshift=.5cm](1) {$A$}; 
        \node at ($(1)!0.5!(4)$) {$=$};
        \node[right of=1](2) {$A'$}; 
        \node[right of=2](2') {$C$}; 
        \node[below of=1](3) {$B$};
        \node[below of=2](4) {$B'$}; 
        \node[below of=2'](4') {$D$}; 

        \draw[->](1) to node[above,la]{$j$} (2); 
        \draw[->](3) to node[below,la]{$k$} (4); 
        \draw[->](2) to node[above,la]{$p$} (2'); 
        \draw[->](4) to node[below,la]{$q$} (4'); 
        \draw[->](1) to node[left,la]{$i$} (3); 
        \draw[->](2) to node[left,la]{$i'$} (4);
        \draw[->](2') to node[left,la]{} (4');
        \pushout{4};
        \pushout{4'};
    \end{tz}
    We need to check that $qk$ is a weak equivalence. Note that $k$ is a $\cJ$-cofibration, and hence a weak equivalence, as it is the pushout of the $\cJ$-cofibration $j$. To see that $q$ is a weak equivalence, note that the pushout $i'$ of $i$ is an $\cI$-cofibration, so in particular a monomorphism by \cref{Icof}. Since $f$ and $j$ are both weak equivalences in $\Dsset_\mathrm{Cat}$ and $C$ is a Segal space, the $2$-out-of-$3$ property implies that the $\cJ$-fibration $p$ is a weak equivalence in $\Dsset_\mathrm{Cat}$ between Segal spaces, and hence a Dwyer--Kan equivalence. By \cref{thm:MSCSS} (iii), the map $p$ is then a weak equivalence in $\Dsset_\mathrm{CSS}$. As the model structure $\Dsset_\mathrm{CSS}$ is left proper, the pushout $q$ of $p$ along the monomorphism $i'$ is a weak equivalence in $\Dsset_\mathrm{CSS}$. By \cref{lem:weinCSStoSegal}, the map $q$ is a weak equivalence in $\Dsset_\mathrm{Cat}$, so that the composite $qk$ is a weak equivalence in $\Dsset_\mathrm{Cat}$ as well.
\end{proof}

\section{Applications}

In this last section, we present two notable consequences of the existence of the categorical model structure. In \cref{sec:nerveofcat}, we show that the inclusion of categories into $(\infty,1)$-categories has a particularly simple description in terms of the model category $\Dsset_\mathrm{Cat}$. Specifically, we prove that the discrete nerve functor is right Quillen. Then, in \cref{sec:homotopylimit}, we show that homotopy limits of $(\infty,1)$-categories can be computed in a manner similar to homotopy limits of ordinary categories. Specifically, we prove that the limit of an injectively fibrant diagram for $\Dsset_\mathrm{Cat}$ is already a homotopy limit for the complete Segal space model structure.

\subsection{Nerves of categories} \label{sec:nerveofcat}

We start by showing that the nerve functor $N\colon \cat\to \Dsset_\mathrm{Cat}$ is right Quillen, where $\cat$ is endowed with the categorical model structure, in which the weak equivalences (resp.~fibrations) are the equivalences of categories (resp.~isofibrations).

\begin{prop} \label{prop:nerve}
   The adjunction from \cref{nerve}
    \[c\colon \Dsset_\mathrm{Cat}\rightleftarrows\mathrm{Cat}\cocolon N \]
    is a Quillen pair, whose derived counit is an equivalence of categories.
\end{prop}
\begin{proof}
    Any map of sets is a fibration in $\sset$, so that any map between simplicial sets (viewed as simplicial spaces) is a Reedy fibration. Since $N\cC$ satisfies the Segal conditions, for every category $\cC$, and the nerve preserves isofibrations, it follows that $N$ preserves fibrations. Next, the trivial fibrations $f\colon \cC\to \cD$ in $\cat$ are the functors that are fully faithful and surjective on objects, so that the induced map $Nf\colon N\cC\to N\cD$ between the nerves satisfies the equivalent conditions of \cref{charDK}. By \cref{proofof4}, the map $Nf$ is then an $\cI$-fibration. This implies that $N$ is a right Quillen functor.

    Finally, the nerve $N\cC$ of any category $\cC$ is cofibrant by \cref{Icof}, so that the derived counit is equivalent to the counit. The latter is an isomorphism as the nerve is fully faithful.
\end{proof}

\subsection{Homotopy limits of \texorpdfstring{$(\infty,1)$}{(infinity,1)}-categories} \label{sec:homotopylimit}

Next, we demonstrate that completing the Segal spaces involved is not required to compute homotopy pullbacks in the complete Segal space model structure.

\begin{prop} \label{prop:pullback}
    Let $f\colon X\to Y$ be an isofibration between Segal spaces, and $Z$ be a Segal space. Then any pullback square in $\Dsset$ of the form
    \begin{tz}
        \node[](1) {$P$}; 
        \node[right of=1,xshift=.2cm](2) {$X$}; 
        \node[below of=1](3) {$Z$}; 
        \node[below of=2](4) {$Y$}; 
        \pullback{1};

        \draw[->](1) to (2); 
        \draw[->](3) to node[below,la]{} (4); 
        \draw[->](1) to (3); 
        \draw[->](2) to node[right,la]{$f$} (4); 
    \end{tz}
    is a homotopy pullback in $\Dsset_{\mathrm{CSS}}$.
\end{prop}

In fact, using \cite[Proposition 19.9.4]{Hirschhorn}, this is a special instance of the following result for more general homotopy limits. 

\begin{prop}
    Let $\cC$ be a small category and consider a diagram $F\colon \cC\to \Dsset$. Then any homotopy limit of $F$ in $\Dsset_\mathrm{Cat}$ is a homotopy limit in $\Dsset_{\mathrm{CSS}}$.
\end{prop}

\begin{proof}
    Consider an injective fibrant replacement $F\Rightarrow\widehat{F}$ in the model structure of diagrams $\cC\to \Dsset_\mathrm{Cat}$, and an injective fibrant replacement $\widehat{F}\Rightarrow \widetilde{F}$ in the model structure of diagrams $\cC\to \Dsset_\mathrm{CSS}$. For every object $c\in \cC$, the induced map $\widehat{F}(c)\to \widetilde{F}(c)$ is a weak equivalence in $\Dsset_\mathrm{CSS}$ between Segal spaces, and so a Dwyer--Kan equivalence by \cref{thm:MSCSS} (iii). As $\id\colon \Dsset_\mathrm{CSS}\to \Dsset_\mathrm{Cat}$ is a right Quillen functor, the diagram $\widetilde{F}\colon \cC\to \Dsset$ is also injectively fibrant as a diagram $\cC\to \Dsset_\mathrm{Cat}$. Hence, we have a levelwise weak equivalence $\widehat{F}\Rightarrow \widetilde{F}$ between injectively fibrant diagrams in the model structure of diagrams $\cC\to \Dsset_\mathrm{Cat}$. As a consequence, the induced map $\mathrm{lim}_\cC \widehat{F}\to \mathrm{lim}_\cC \widetilde{F}$ is a Dwyer-Kan equivalence between Segal spaces, and so a weak equivalence in $\Dsset_\mathrm{CSS}$ by \cref{thm:MSCSS} (iii). This concludes the proof. 
\end{proof}

\begin{rem}
Since $\Dsset_\mathrm{Cat}$ is cartesian closed by \cref{thm:cartclosed}, every Segal space $X$ has a canonical Reedy fibrant simplicial resolution $\Delta^{\op}\to \Dsset_\mathrm{Cat}$ sending $[m]\mapsto X^{N I[m]}$, where $NI[m]$ is the nerve of the contractible groupoid with $(m+1)$ objects. Indeed, this follows by adjunction from the fact that $[m]\mapsto N I[m]$ defines a Reedy cofibrant cosimplicial diagram in $\Dsset_\mathrm{Cat}$ that is homotopically constant. One can use this canonical simplicial resolution to explicitly compute the homotopy limit of a diagram of Segal spaces $X\colon \cC\to \Dsset_\mathrm{Cat}$, which does not need to be injectively fibrant. This can be expressed as a weighted limit, via the Bousfield--Kan formula \cite{AO}: writing $X^K=\lim_{[m]\in (\Delta_{/K})^{\mathrm{op}}} X^{NI[m]}$, we have that
\[
\textstyle \mathrm{holim}_{\cC} X = \int_{c\in \cC} X(c)^{N(\cC/c)}.
\]
For example, an object in $\mathrm{holim}_{c\in \cC} X(c)$ can be described as a matching family of objects up to coherent homotopy, consisting of the data of an object $x_c\in X(c)_{0, 0}$ for each $c\in\cC$, and a compatible family of maps $f_\alpha\colon N I[m]\to X(c_m)$ for each chain $\alpha\colon c_0\to c_1\to \dots \to c_m$ in $\cC$. 
\end{rem}

\bibliographystyle{alpha}
\bibliography{references}

\begin{thebibliography}{GMSV23}

\bibitem[AF20]{AF}
David Ayala and John Francis.
\newblock Fibrations of {$\infty$}-categories.
\newblock {\em High. Struct.}, 4(1):168--265, 2020.

\bibitem[A{\O}23]{AO}
Sergey Arkhipov and Sebastian {\O}rsted.
\newblock Homotopy (co)limits via homotopy (co)ends in general combinatorial
  model categories.
\newblock {\em Appl. Categ. Structures}, 31(6):Paper No. 47, 10, 2023.

\bibitem[Ber07a]{Bergnerfibrant}
Julia~E. Bergner.
\newblock A characterization of fibrant {S}egal categories.
\newblock {\em Proc. Amer. Math. Soc.}, 135(12):4031--4037, 2007.

\bibitem[Ber07b]{Bergner}
Julia~E. Bergner.
\newblock Three models for the homotopy theory of homotopy theories.
\newblock {\em Topology}, 46(4):397--436, 2007.

\bibitem[BV73]{BV}
J.~Michael Boardman and Rainer~M. Vogt.
\newblock {\em Homotopy invariant algebraic structures on topological spaces},
  volume Vol. 347 of {\em Lecture Notes in Mathematics}.
\newblock Springer-Verlag, Berlin-New York, 1973.

\bibitem[CS19]{CS}
Damien Calaque and Claudia Scheimbauer.
\newblock A note on the {$(\infty,n)$}-category of cobordisms.
\newblock {\em Algebr. Geom. Topol.}, 19(2):533--655, 2019.

\bibitem[GMSV23]{GMSV}
L\'eonard Guetta, Lyne Moser, Maru Sarazola, and Paula Verdugo.
\newblock Fibrantly-induced model structures.
\newblock Preprint on
  \href{https://arxiv.org/abs/2301.07801}{arXiv:2301.07801}, 2023.

\bibitem[Hau17]{Haugseng}
Rune Haugseng.
\newblock The higher {M}orita category of {$\mathbb{E}_n$}-algebras.
\newblock {\em Geom. Topol.}, 21(3):1631--1730, 2017.

\bibitem[Hir03]{Hirschhorn}
Philip~S. Hirschhorn.
\newblock {\em Model categories and their localizations}, volume~99 of {\em
  Mathematical Surveys and Monographs}.
\newblock American Mathematical Society, Providence, RI, 2003.

\bibitem[HS98]{HS}
Andr\'e Hirschowitz and Carlos Simpson.
\newblock Descente pour les n-champs.
\newblock Preprint on
  \href{https://arxiv.org/abs/math/9807049}{arXiv:math/9807049}, 1998.

\bibitem[Joy02]{Joyal}
Andr\'e{} Joyal.
\newblock Quasi-categories and {K}an complexes.
\newblock {\em J. Pure Appl. Algebra}, 175(1-3):207--222, 2002.
\newblock Special volume celebrating the 70th birthday of Professor Max Kelly.

\bibitem[Joy08]{JoyalVolumeII}
Andr\'e{} Joyal.
\newblock The theory of quasi-categories and its applications.
\newblock preprint available at
  \url{http://mat.uab.cat/~kock/crm/hocat/advanced-course/Quadern45-2.pdf},
  2008.

\bibitem[JT07]{JT}
Andr\'e{} Joyal and Myles Tierney.
\newblock Quasi-categories vs {S}egal spaces.
\newblock In {\em Categories in algebra, geometry and mathematical physics},
  volume 431 of {\em Contemp. Math.}, pages 277--326. Amer. Math. Soc.,
  Providence, RI, 2007.

\bibitem[Lur09a]{Lurie}
Jacob Lurie.
\newblock {\em Higher topos theory}, volume 170 of {\em Annals of Mathematics
  Studies}.
\newblock Princeton University Press, Princeton, NJ, 2009.

\bibitem[Lur09b]{LurieTFT}
Jacob Lurie.
\newblock On the classification of topological field theories.
\newblock In {\em Current developments in mathematics, 2008}, pages 129--280.
  Int. Press, Somerville, MA, 2009.

\bibitem[Nui16]{Nuiten}
Joost Nuiten.
\newblock Localizing $\infty$-categories with hypercovers.
\newblock Preprint on
  \href{https://arxiv.org/abs/1612.03800}{arXiv:1612.03800}, 2016.

\bibitem[Pel02]{Pellissier}
Regis Pellissier.
\newblock Cat\'egories enrichies faibles.
\newblock thesis available at
  \url{https://theses.hal.science/tel-00003273/fr/}, 2002.

\bibitem[Rez01]{Rezk}
Charles Rezk.
\newblock A model for the homotopy theory of homotopy theory.
\newblock {\em Trans. Amer. Math. Soc.}, 353(3):973--1007, 2001.

\bibitem[Sim12]{Simpson}
Carlos Simpson.
\newblock {\em Homotopy theory of higher categories}, volume~19 of {\em New
  Mathematical Monographs}.
\newblock Cambridge University Press, Cambridge, 2012.

\end{thebibliography}

\end{document}